\documentclass[11pt,aps,prb]{article}

\pdfoutput=1

\usepackage{mystyle}
\usepackage{url}

\begin{document}
\title{An Interpolating Distance between\\Optimal Transport and Fisher-Rao}

\author{
\begin{tabular}{c}
	Lenaic Chizat \qquad Bernhard Schmitzer \\
	Gabriel Peyr\'e \qquad François-Xavier Vialard\\[2mm]
	Ceremade, Université Paris-Dauphine \\
	\texttt{\small\{chizat,peyre,schmitzer,vialard\}@ceremade.dauphine.fr}
\end{tabular} }

\maketitle


\begin{abstract}
This paper defines a new transport metric over the space of non-negative measures. This metric interpolates between the quadratic Wasserstein and the Fisher-Rao metrics and generalizes optimal transport to measures with different masses.
It is defined as a generalization of the dynamical formulation of optimal transport of Benamou and Brenier, by introducing a source term in the continuity equation. The influence of this source term is measured using the Fisher-Rao metric, and is averaged with the transportation term. This gives rise to a convex variational problem defining our metric. 
Our first contribution is a proof of the existence of geodesics (i.e. solutions to this variational problem).
We then show that (generalized) optimal transport and Fisher-Rao metrics are obtained as limiting cases of our metric.
Our last theoretical contribution is a proof that geodesics between mixtures of sufficiently close Diracs are made of translating mixtures of Diracs. 
Lastly, we propose a numerical scheme making use of first order proximal splitting methods and we show an application of this new distance to image interpolation.
\end{abstract}


\section{Introduction}

Optimal transport is an optimization problem which gives rise to a popular class of metrics between probability distributions. We refer to the monograph of Villani~\cite{cedric2003topics} for a detailed overview of optimal transport. We now describe in an informal way various formulations of transportation-like distances, and expose also our new metric. A mathematically rigorous treatment can be found in Section~\ref{sec:theory}.

\subsection{Static and Dynamic Optimal Transport}

In its original formulation by Monge, optimal transport takes into account a ground cost that measures the amount of work needed to transfer a measure towards another one.  This formulation was later relaxed by Kantorovich~\cite{Kantorovich42} as a convex linear program, often referred to as a ``static formulation''. 
Given a \emph{ground cost} $c : \Om^2 \mapsto \R$, $\Om \subset \R^d$ and two probability measures $\rho_0$ and $\rho_1$ on $\Om$, the central problem is that of minimizing
\begin{equation}
\int_{\Om \times \Om} c(x,y) \,\d\gamma (x,y)
\label{optimal transport problem}
\end{equation}
over the set $\Gamma_{(\rho_0,\rho_1)}$ of coupling measures, that is non-negative measures $\gamma$ on $\Om^2$ admitting $\rho_0$ and $\rho_1$ as first and second marginals. More explicitly, for every measurable set $A \in \Om$, $\gamma(A,\Om)=\rho_0(A)$ and $\gamma(\Om,A)=\rho_1(A)$. When the \emph{ground cost} is the distance $|x-y|^p$, the minimum induces a metric on the space of probability measures which is often referred to as the p-Wasserstein distance, denoted by $W_p$.

In a seminal paper, Benamou and Brenier~\cite{benamou2000computational} showed that this static formulation is equivalent to a ``dynamical'' formulation, where one seeks for a ``geodesic'' $\rho(t,x)$ that evolves in time between the two densities. More precisely, they showed that the squared 2-Wasserstein is obtained by minimizing the kinetic energy
\[
\int_0^1 \int_{\Omega} |v(t,x)|^2 \rho(t,x) \d x \d t
\]
where $\rho$ interpolates between $\rho_0$ and $\rho_1$ and $v$ is a velocity field which advects the mass distribution $\rho$. Introducing the key change of variables $(\rho,v) \mapsto (\rho,\M)\defeq (\rho,\rho v)$, this amounts to minimizing the convex functional
\begin{equation}\label{eq-dynamic}
 	\int_0^1 \int_{\Omega} \frac{|\omega(t,x)|^2}{\rho(t,x)} \,\d x \d t
\end{equation}
over the couples $(\rho,\omega)$ satisfying the continuity equation 
\begin{equation}\label{eq-continuity}
	\partial_t \rho + \nabla \cdot \omega =0, \quad \rho(0,\cdot) = \rho_0, \quad \rho(1,\cdot) = \rho_1\,.
\end{equation}
In particular, their paper emphasized the Riemannian interpretation of this metric and paved the way for efficient numerical solvers.

The geometric nature of optimal transport makes this tool interesting to perform shape interpolation, especially in medical image analysis~\cite{haker2004optimal}. Even though the resulting interpolation is somewhat simplistic compared to methods based on diffeomorphic flows such as LDDMM~\cite{beg2005computing}, the fact that the problem is convex is interesting for numerical purposes and large scale applications. 
A major restriction of optimal transport is however that it is only defined between measures having the same mass. This might be an issue for applications in medical imaging, where measures need not be normalized, and biophysical phenomena at stake typically involve some sort of mass creation or destruction.
To date, several generalizations of optimal transport have been proposed, with various goals in mind (see Section~\ref{previous works}). However, there is currently no formulation for which geodesics can be interpreted as a joint  displacement and modification of mass. In this article, we propose to tackle this issue by interpolating between the Wasserstein and the Fisher-Rao metric, which can be achieved using a convex dynamical formulation.

In the remainder of this section, we introduce our model and describe related work which generalizes optimal transport to unbalanced measures. In addition, we introduce in an informal manner a family of distances which shows the similarity between \emph{a priori} unrelated models of transport between unbalanced measures. 
In Section~\ref{sec:theory}, we prove the existence of geodesics and exhibit optimality conditions. 
Section~\ref{sec-interpolation} is devoted to the study of the effect of varying the interpolation factor and to the limit models. 
We detail in Section~\ref{sec-explicit-geod} explicit geodesics between atomic measures satisfying suitable properties. 
Finally, we describe a numerical scheme in Section~\ref{sec-numerics} and use the distance for image interpolation experiments.

\subsection{Presentation of the Model and Contributions}

The dynamical formulation~\eqref{eq-dynamic} suggests a natural way to relax the equality of mass constraint by introducing a source term $\zeta$ in the continuity equation~\eqref{eq-continuity} as follows
\[
	\D_t \rho + \nabla \cdot \M = \zeta, \quad 
	\rho(0,\cdot) = \rho_0, \quad \rho(1,\cdot) = \rho_1, 
\]
and adding a term penalizing $\zeta$ in the Benamou-Brenier functional, in the spirit of metamorphoses introduced in imaging in \cite{Metamorphosis2005}. We think it is natural to require the following two properties for the penalization on $\Z$:
\begin{enumerate}
	\item \textit{Reparametrization invariance.} As the geometrical aspect of the problem is taken care of by the optimal transport part, it is important for the penalty defined on $\zeta$ not to be sensitive to deformations of the objects.
	\item \textit{Relation with  a Riemannian metric.} Riemannian metrics are natural for interpolation purposes and for applications to imaging. Moreover, the standard notion of curvature is available in a Riemannian setting and give some important informations on the underlying geometry.
\end{enumerate}
It has been shown recently that the only Riemannian metric which is invariant by reparametrization on the space of smooth positive densities on a closed manifold of dimension greater than one is the Fisher-Rao metric~\cite{Bauer:2014aa}. It was introduced in 1945 by C. R. Rao, who applied for the first time differential geometry in the field of statistics~\cite{radhakrishna1945information}. The squared Fisher-Rao distance between two smooth positive densities $\rho_0$ and $\rho_1$ is defined as the minimum of
\[
 \int_0^1 \int_{\Om} \frac{\Z(t,x)^2}{\rho(t,x)}\, \d x \d t 
\]
over the couples $(\rho,\Z)$ such that $\D_t \rho = \Z$, $\rho(0,\cdot)=\rho_0$ and $\rho(1,\cdot)=\rho_1$. Its geodesics are explicit and given by $\rho(t,x) = \left[ t\sqrt{\rho_1(x)}+(1-t)\sqrt{\rho_0(x)}\right]^2$. Note that this metric can be rigorously defined on the space of non-negative measures (see~\thref{def FR}), but it is not Riemannian anymore.

The considerations above lead us to the problem of minimizing
\begin{equation}
 \int_0^1 \left[ \int_{\Om} \frac{|\M(t,x)|^2}{2\rho(t,x)} \d x 
 + \delta^2 \int_{\Om} \frac{\Z(t,x)^2}{2\rho(t,x)} \d x\right] \d t
\label{defsmooth}
\end{equation}
under the constraints  
\begin{equation}
 	\D_t \rho + \nabla \cdot \M = \Z,  \quad \rho(0,\cdot)=\rho_0, \quad \rho(1,\cdot)=\rho_1, 
 \label{continuity with source intro}
\end{equation}
where $\delta>0$ is a parameter. 
The metric defined by this infimum as well as the minimizers are the central objects of this paper. While formula \eqref{defsmooth} only makes sense for positive smooth densities, we rigorously extend this definition on the space of non-negative measures in Section \ref{sec:theory}. This extension is well behaved since the functional that is minimized is convex and positively 1-homogeneous, two key properties which allow to extend it readily as a lower semi-continuous (l.s.c.) functional over measures~\cite{bouchitte1990new}.

Let us now give a physical interpretation of this model: just as the first term in the functional measures the kinetic energy of the interpolation $\Vert v \Vert^2_{L^2(\rho)} $ where $v$ is the velocity field, the second term measures an energy of growth $\Vert g \Vert^2_{L^2(\rho)} $ where $g$ denotes the rate of growth $\Z/\rho$ which is a scalar field depending on $t$ and $x$. In this way, our problem can be rewritten as the minimization of
\[
\frac12 \int_0^1 \left[ \int_{\Omega} |v(t,x)|^2 \rho(t,x) \d x  + \delta^2 \int_{\Omega} |g(t,x)|^2 \rho(t,x) \d x \right] \d t
\]
under the constraints
\[
\D_t \rho = g\rho - \nabla \cdot (\rho v) , \quad \rho(0,\cdot) = \rho_0 , \quad  \rho(1,\cdot) = \rho_1\, .
\]
This formulation emphasizes the geometric nature of this model.

\subsection{Previous Works and Connections}
\label{previous works}

The idea of extending optimal transport to unbalanced measures is far from new, and is already suggested by Kantorovitch in 1942 (see~\cite{guittet2002extended}) where he allows mass to be thrown away out of the (compact) domain. The geometry of the domain also plays a role in the more recent models of~\cite{figalli2010new} where the boundary of the domain is taken as an infinite reserve of mass. In the past few years, some authors combined generalizations of optimal transport with numerical algorithms in order to tackle practical applications. In~\cite{benamou2003numerical}, the problem is formulated as an optimal control one where the \emph{matching} term is the $L^2$ distance between densities. The dynamic problem remains convex and is solved numerically with an augmented Lagrangian algorithm.
A quadratic penalization is also suggested in~\cite{OTmaasrumpf} but based, this time, on the dynamic formulation, with the model of metamorphosis is mind: they use constraint  \eqref{continuity with source intro} and add the term $\iint \Z^2 \d x \d t$ in the Benamou-Brenier functional (along with a viscous dissipation term). 
From the perspective of computing image interpolations~\cite{lombardi2013eulerian} suggest, instead of penalizing $\Z$ in the dynamic functional, to fix it as a function of $\rho$, according to some \emph{a priori} on the growth model. Then, if the \emph{a priori} is well chosen, minimizing the Benamou-Brenier functional gives rise to meaningful interpolations. Note that one of our limit models (see \thref{def gBB}) falls into this category.

Now we introduce with more details the classes of distances introduced in \cite{piccoli2014generalized, piccoli2013properties}, along with partial optimal transport distances, as they share some similarities with our approach.

\paragraph{Partial transport and its dynamical formulation.}

Partial optimal transport, introduced in~\cite{caffarelli2010free} consists in transferring a fixed amount of mass between two measures $\rho_0$ and $\rho_1$ as cheaply as possible.
Letting $\Gamma_{\leq(\rho_0,\rho_1)}$ be the set of measures on $\Om \times \Om$ whose left and right marginals are dominated by $\rho_0$ and $\rho_1$ respectively, this amounts to minimizing \eqref{optimal transport problem} under the constraints $\gamma \in \Gamma_{\leq(\rho_0,\rho_1)}$ and $\vert \gamma \vert (\Om \times \Om)=m$. 
What happens is rather intuitive: the plan is divided between an active set, where classical optimal transport is performed, and an inactive set. A theoretical study of this model is done in~\cite{caffarelli2010free} and~\cite{figalli2010optimal}, where for instance results on the regularity of the boundary between the active an inactive sets are proved.
For the quadratic cost, the Lagrangian formulation for this problem~\cite[Corollary 2.11]{caffarelli2010free} reads
\begin{equation}
\inf_{\gamma \in \Gamma_{\leq(\rho_0,\rho_1)}} \int_{\Om \times \Om} \left( \vert x-y\vert^2 -\lambda \right) \d\gamma (x,y) \, .
\label{PTdef}
\end{equation}
where choosing the Lagrange multiplier $\lambda$ amounts to choosing the maximum distance mass can be transported---which is $\sqrt{\lambda}$ ---and consequently, the amount $m$ of mass transported. 

Independently, a class of distances interpolating between total variation and the $W_p$ metrics has been introduced by ~\cite{piccoli2014generalized, piccoli2013properties}, taking their inspiration from the optimal control formulation of~\cite{benamou2001mixed}. 
This distances are defined as the infimum of
\[
W_p (\tilde{\rho}_0,\tilde{\rho}_1) + 
\delta \cdot
(  \vert \rho_0 - \tilde{\rho}_0 \vert + \vert \rho_1 - \tilde{\rho}_1 \vert  )
\]
over $\tilde{\rho}_0, \tilde{\rho}_1$, non-negative measures on $\Om$, for $p \geq 1$. Now remark that, for fixed $\rho_0$ and $\rho_1$ and up to changing $\delta$, the previous problem is equivalent to minimizing
\begin{equation}
W^p_p (\tilde{\rho}_0,\tilde{\rho}_1) + 
\delta^p \cdot
(  \vert \rho_0 - \tilde{\rho}_0 \vert + \vert \rho_1 - \tilde{\rho}_1 \vert  )
\label{PRdef}
\end{equation}
To our knowledge, the equivalence between this model and partial optimal transport has not been exposed yet although this allows an interesting dynamical formulation for the well studied partial transport problem. This is the object of the following proposition.

\begin{proposition}[Equivalence between partial OT and Piccoli-Rossi distances]
\thlabel{equivalence TV partial}
When choosing $2\delta^2=\lambda$ and $p=2$, problems \eqref{PTdef} and \eqref{PRdef} are equivalent.
\end{proposition}

\begin{proof}
First, it is clear that problem \eqref{PRdef} is left unchanged when adding the domination constraints $\tilde{\rho}_0\leq \rho_0$ and $\tilde{\rho}_1\leq \rho_1$ (\cite[Proposition 4]{piccoli2013properties}). This implies that $\vert \rho_i-\tilde{\rho_i}\vert$ is equal to $\rho_i(\Om)-\tilde{\rho_i}(\Om)$ in such a way that \eqref{PRdef} can be rewritten as
\[
\inf_{\gamma \in \Gamma_{\leq(\rho_0,\rho_1)}} \int_{\Om \times \Om} \left( \vert x-y\vert^2 \right) \d\gamma (x,y) +
\delta^2
\left(  \rho_0(\Om)+\rho_1(\Om)-2\gamma(\Om\times \Om )\right)
\]
This is clearly equivalent to the Lagrangian formulation \eqref{PTdef}.
\end{proof}

The interest of this result is that it allows a dynamic formulation of the partial optimal transport problem. Indeed, it is shown in~\cite{piccoli2013properties} that minimizing \eqref{PRdef} is equivalent to minimizing the following dynamic functional
\begin{equation}
\label{dynamic PT}
\int_0^1 \left[ \int_{\Om} \frac{|\M(t,x)|^2}{\rho(t,x)} \d x + \delta^2 \int_{\Om} |\Z(t,x)| \d x \right] \d t
\end{equation}
subject to the continuity constraints with a source \eqref{continuity with source intro}.

\paragraph{The case $W_1$.}

When chosing $p=1$ in \eqref{PRdef}, one obtains another partial optimal transport problem, with $|\cdot|$ as a ground metric. This distance was already introduced by~\cite{hanin1992kantorovich} in view of the study of Lipschitz spaces. Without providing a proof, it is quite clear that this distance is also obtained by minimizing
\begin{equation}
\int_0^1 \left[ \int_{\Om} |\M(t,x)| \d x + \delta \int_{\Om} |\Z(t,x)| \d x \right] \d t
\label{eq: Hanin KR}
\end{equation}
over the set set of solutions to \eqref{continuity with source intro}.

\paragraph{A unifying framework: homogeneity and convexity.}

We conclude this bibliographical section by suggesting a unifying framework for a class of generalizations of the $W_p$ metrics. Consider the problem of minimizing
\[
 \int_0^1 \left[ \frac1p \int_{\Om} \frac{|\M|^p}{\rho^{p-1}} \d x + \delta^p \frac1q \int_{\Om} \frac{|\zeta|^q}{\rho^{q-1}}\d x \right] \d t
\]
over the set of solutions to  \eqref{continuity with source intro}. The functional penalizes the $p$-th power of the $L^p_{\rho}$ norm of the velocity field and the $q$-th power of the $L^q_{\rho}$ norm of the rate of growth, the latter being reparametrization invariant. Whatever $p,q\geq 1$, the above functional can be rigorously defined as a homogeneous, convex, l.s.c functional on measures (as in Section \ref{sec:theory}). Also, this problem can be seen as the dual problem of maximizing the variation of an energy
\[
\int_{\Om} \varphi(1,\cdot) \rho_1 - \int_{\Om} \varphi(0,\cdot) \rho_0
\]
over continuously differentiable potentials $\varphi$ on $[0,1]\times \Om$ satisfying
\[
\D_t \varphi + \frac{p-1}{p}|\nabla \varphi |^{p/(p-1)} + \delta^{-p/(q-1)} \frac{q-1}{q}|\varphi|^{q/(q-1)} \leq 0
\]
with some adaptations when $p$ or $q$ takes the value $1$ or $+ \infty$.

For $p=q=1$, we recover \eqref{eq: Hanin KR} and for $p=2$ and $q=1$, we recover \eqref{dynamic PT}, implying the connections to partial optimal transport arising from our previous remarks. While there is an interesting decoupling effect between the growth term and the transport term when $p$ or $q$ is equal to $1$, one should choose in general $p=q\geq1$ so as to define (the $p$-th power of) a metric. In this article, we focus on the case $p=2,\, q=2$, which is the only Riemannian-like metric of this class, namely our interpolated distance between $W_2$ and Fisher-Rao.

\subsection{Relation with~\cite{new2015kondratyev}}

After finalizing this paper, we became aware of the recent work of~\cite{new2015kondratyev} which defines and studies the same metric.  These two works were done simultaneously and independently.  The approach of~\cite{new2015kondratyev} is however quite different.
While both works prove existence of geodesics, the proofs rely on different strategies, and in particular our construction is based on convex analysis.
The use of tools from convex analysis allows us to prove a useful sufficient condition for the uniqueness of a geodesic.
We study the limiting models (generalized transport and Fisher-Rao), through the introduction of the $\delta$ parameter. 
We prove existence and uniqueness of travelling Diracs solutions, while~\cite{new2015kondratyev} provides an informal discussion. 
Lastly, we detail a numerical scheme and show applications to image interpolation. 
Let us stress that~\cite{new2015kondratyev} provides many other contributions (such as a Riemannian calculus, topological properties and the definition of gradient flows) that we do not cover here. 


\subsection{Notations and Preliminaries}

In the sequel, we consider an open bounded convex spatial domain $\Om \subset \R^d$, $d \in \mathbb{N}^*$. The set of Radon measures (resp. positive Radon measures) on $X$ ($X$ is typically $\Om$ or $[0,1]\times \Om$), is denoted by $\mathcal{M}(X)$ (resp. $\mathcal{M}_+(X)$). Endowed with the total variation norm
\[
|\mu|(X) \defeq \sup \left\{ \int_{X} \varphi \d \mu : \varphi \in C(X), \vert \varphi \vert_{\infty} =1 \right\} ,
\]
$\mathcal{M}(X)$ is a Banach space. We identify this space with the topological dual of the space of real valued continuous functions $C(X)$ endowed with the $\sup$ norm topology. Another useful topology on $\mathcal{M}(X)$ is the weak* topology arising from this duality: a sequence of measures $(\mu_n)_{n\in \N}$ weak* converges towards $\mu \in \mathcal{M}(X)$ if and only if for all $\varphi \in C(X)$, $\lim_{n\rightarrow + \infty}\int_{X} \varphi \d\mu_n = \int_{X} \varphi \d\mu$. Any bounded subset of $\mathcal{M}(X)$ (for the total variation) is relatively sequentially compact for the weak* topology.
\\

\noindent We also use the following notations:\\
\indent$\bullet$ $\mu \ll \nu$ means that the measure $\mu$ is absolutely continuous w.r.t. $\nu$;
\\
\indent$\bullet$ for a (possibly vector) measure $\mu$, $|\mu| \in \mathcal{M}_+(X)$ is its variation;
\\
\indent$\bullet$ for a positively homogeneous function $f$, we denote by $f(\mu)$ the measure defined by $f(\mu)(A) = \int_A f(\d\mu/\d\lambda) \d\lambda$ for any measurable set $A$, with $\lambda$ satisfying $|\mu| \ll \lambda$. The homogeneity assumption guarantees that this definition does not depend on $\lambda$;
\\
\indent$\bullet$ $T_{\#} \mu$ is the image measure of $\mu$ through the measurable map $T: X_1 \to X_2$, also called the pushforward measure. It is given by $T_{\#} \mu (A_2) = \mu(T^{-1}(A_2))$;
\\
\indent$\bullet$ $\delta_x$ is a Dirac measure of mass $1$ located at the point $x$;
\\
\indent$\bullet$ $\iota_{\mathcal{C}}$ is the (convex) indicator function of a convex set $\mathcal{C}$ which takes the value $0$ on $\mathcal{C}$ and $+\infty$ everywhere else;
\\
\indent$\bullet$ if $f$ is a convex function on a normed vector space $E$ with values in $]-\infty, + \infty]$, $f^*$ is its dual function in the sense of Fenchel-Legendre duality, i.e.\ for $y$ in the topological dual (or pre-dual) space $E'$, $f^*(y) = \sup_{x\in E} \langle x, y \rangle - f(x) $. The subdifferential of $f$ is denoted $\D f$ and is the set valued map $x \mapsto \{ y\in E' : \forall x' \in E, \, f(x')-f(x) \geq \langle y, x'-x \rangle \}$.
\\

Finally, when $\mu$ is a measure on $[0,1] \times \Om$ which time marginal is absolutely continuous with respect to the Lebesgue measure on $[0,1]$, we denote by $(\mu_t)_{t\in[0,1]}$ the ($\d t$-a.e.\ uniquely determined) family of measures satisfying $\mu = \int_0^1 (\mu_t \otimes \delta_t ) \d t$, where $\otimes$ denotes the product of measures. This is a consequence of the disintegration theorem (see for instance ~\cite[Theorem 5.3.1]{ambrosio2006gradient}) which we frequently use without explicitly mentioning it. In order to alleviate notations, we replace the integral expression of $\mu$ by $\mu = \mu_t \otimes \d t$.



\section{Definition and Existence of Geodesics}
\label{sec:theory}

In order to prove that the minimum of the variational problem is attained, rather than using the direct method of calculus of variation, we choose to take advantage of the convex nature of our problem and prove it by Fenchel-Rockafellar duality. Indeed, we show that (a generalization of) problem  \eqref{defsmooth} is obtained by considering the dual of a variational problem on the space of differentiable functions. This point of view allows to kill two birds with one shot: proving the existence of geodesics and exhibiting sufficient optimality conditions.

\subsection{Definition of the Interpolating Distance}
\label{subsec:average}

Consider the closed convex set
\[
B_{\delta} \defeq \left\{ (a,b,c)\in \R \times \R^d \times \R : a+\frac{1}{2}\left( |b|^2+\frac{c^2}{\delta^2} \right) \leq 0 \right\}
\]
and its convex indicator function
\[
\iota_{B_{\delta}} : (a,b,c) \in \R^{d+2} \mapsto 
\begin{cases}
0 & \text{if } (a,b,c) \in B_{\delta} \\
+ \infty & \text{otherwise.}
\end{cases}
\]
We denote by $\fonc$ the Fenchel-Legendre conjugate of $\iota_{B_{\delta}}$, i.e.
\[
\fonc : (x,y,z)\in \R\times\R^d\times\R \mapsto
\begin{cases}
\frac{\vert y \vert^2 + \delta^2 z^2 }{2x} &\text{if $x>0$,}\\
0 & \text{if } (x,|y|,z)=(0,0,0)\\
+\infty & \text{otherwise}
\end{cases}
\]
which is by construction proper, convex, lower semicontinuous (l.s.c.) and homogeneous. One can check by direct computations that
\begin{equation}
\label{subdiff fonctional}
\D \fonc (x,y,z) =
\begin{cases}
\left( -\frac{|y|^2+\delta^2 z^2}{2x^2}, \frac{y}{x},\frac{\delta^2 z}{x} \right)& \text{if $x>0$,}\\
B_{\delta}& \text{if $(x,|y|,z)=(0,0,0)$,}\\
\emptyset & \text{otherwise}.
\end{cases}
\end{equation}
We denote $\mathcal{M} \defeq \mathcal{M}([0,1]\times \Om)$ and for $\mu=(\rho,\M,\Z) \in \mathcal{M}\times \mathcal{M}^d \times \mathcal{M}$ we define the convex functional
\[
\ifonc(\mu) \defeq \int_{[0,1] \times \Om} \fonc\left(\frac{\d\mu}{\d\lambda}\right) \d\lambda
\]
where $\lambda$ is any non-negative Borel measure satisfying $\vert \mu \vert \ll \lambda$. As $\fonc$ is homogeneous, the definition of $\ifonc$ does not dependent of the reference measure $\lambda$.

\begin{definition}[Continuity equation with source]
\thlabel{continuity equation}
We denote by $\ccons$ the affine subset of $\mathcal{M}\times \mathcal{M}^d \times \mathcal{M}$ of triplets of measures $\mu=(\rho,\M,\Z)$ satisfying the continuity equation $\D_t \rho + \nabla \cdot \M = \Z$ weakly, interpolating between $\rho_0$ and $\rho_1$ and satisfying homogeneous Neumann boundary conditions. This is equivalent to requiring
\begin{equation}
\label{continuity weak}
\int_0^1 \int_{\Om} \D_t \varphi \ \d\rho + \int_0^1 \int_{\Om} \nabla \varphi \cdot \d\M + \int_0^1 \int_{\Om} \varphi \ \d\Z = \int_{\Om} \varphi(1,\cdot)\d\rho_1 - \int_{\Om} \varphi(0,\cdot)\d\rho_0
\end{equation}
for all $\varphi \in C^1([0,1]\times \bar{\Om})$.
\end{definition}
\begin{remarks}\mbox{}
\begin{itemize}
\item The set $\ccons$ is not empty: it contains for instance the linear interpolation $(\rho,0,\D_t \rho)$ with $\rho = (t \rho_1 + (1-t) \rho_0 ) \otimes \d t$. Moreover, if we assume that $\rho_0$ and $\rho_1$ have equal mass, then there exists $(\rho,\M,\Z) \in \ccons$ such that $\zeta =0$. This is a consequence of \cite[Theorem 8.3.1]{ambrosio2006gradient}.
\item If we make the additional assumption that $\M\ll\rho$ and $\Z \ll \rho$ (which is satisfied as soon as $\ifonc(\mu)$ is finite), we find that the time marginal of $\mu$ is absolutely continuous w.r.t.\ the Lebesgue measure, allowing to disintegrate it in time.
\item Under the same assumptions, $\mu$ admits a continuous representative, i.e.\ it is $\d t-$a.e.\ equal to a curve which is weak* continuous in time. 
\item When $\rho$ is disintegrable in time, the map $t \mapsto \rho_t(\Om) = \int_{\Om} \rho_t$ admits the distributional derivative $\rho'_t (\Om) = \zeta_t(\Om)$, for almost every $t\in[0,1]$. Indeed, taking a test function $\varphi$ constant in space in \eqref{continuity weak} gives
\[
\int_0^1 \varphi'(t) \rho_t(\Om)\d t + \int_0^1 \varphi(t) \Z_t(\Om)\d t = \varphi(1)\rho_1(\Om) -\varphi(0)\rho_0(\Om)\, .
\]
\end{itemize}
\end{remarks}
We are now in position to define the central object of this paper $\WF_{\delta}$, as a map $\mathcal{M}_+(\Om) \times \mathcal{M}_+(\Om) \rightarrow \R_+$.
\begin{definition}[Interpolating distance]
\thlabel{interpolating distance} 
For $\rho_0$ and $\rho_1$ in $\mathcal{M}_+(\Om) \times \mathcal{M}_+(\Om)$, the (squared) interpolating distance is defined as
\begin{equation*}
\label{dual}
\WF_{\delta}^2(\rho_0,\rho_1) \defeq \inf_{\mu \in \ccons} \ifonc (\mu) \, . \tag{$\mathcal{P}_{\mathcal{M}}$}
\end{equation*}
\end{definition}

The following result shows that $\WF_{\delta}(\rho_0,\rho_1)$ is always finite. 

\begin{proposition}[Bound on the distance]
\thlabel{bounddistance}
Let $\rho_0$ and $\rho_1$ be in $\mathcal{M}_+(\Om)$. The following bounds hold:
\begin{align*}
\WF^2_{\delta}(\rho_0,\rho_1) & \leq 2\delta^2 \vert (\sqrt{\rho_1} -\sqrt{\rho_0})^2 \vert(\Om) \\
&\leq 2\delta^2(\rho_0(\Om)+\rho_1(\Om)).
\end{align*}
\end{proposition}

\begin{proof}
Consider the triplet of measures $\mu=(\rho,\M,\Z)$ with
\[
\begin{cases}
\rho =  \left( t\sqrt{\rho_1}+(1-t)\sqrt{\rho_0} \right)^2  \otimes \d t \, ,\\
\M = 0\, , \\
\Z = 2(\sqrt{\rho_1}-\sqrt{\rho_0})(t\sqrt{\rho_1}+(1-t)\sqrt{\rho_0}) \otimes \d t  ,
\end{cases}
\]
and notice that it belongs to $\ccons$.  Taking $\lambda \in \mathcal{M}_+$ such that $\mu \ll \lambda$, the first bound comes from
 \[
 \ifonc(\mu)= 2 \delta^2 \int_{\Om} \left( \sqrt{\frac{\d\rho_1}{\d\lambda}} -  \sqrt{\frac{\d\rho_0}{\d\lambda}}\right)^2 \d\lambda
 \]
 and does not depend on $\lambda$ by homogeneity, while the second bound is obtained in the worst case when the supports are disjoint.
 \end{proof}

\begin{remark}
Notice that the first bound is tight and is actually the Fisher-Rao metric, defined in \thref{def FR}. It is the solution of \eqref{dual}  when adding the constraint $\M=0$ (see \thref{uniqueness FR}).
\end{remark}

\subsection{Existence of Geodesics}
We now prove the existence of geodesics. The proof mainly uses the Fenchel-Rockafellar duality theorem---which can be found in its canonical form along with its proof in \cite[Theorem 1.9]{cedric2003topics}---and a duality result for integrals of convex functionals.

\begin{definition}[Geodesics]
Geodesics are measures $\rho \in \mathcal{M}_+$ for which there exists a momentum $\M$ and a source $\Z$ such that $(\rho,\M,\Z)$ is a minimizer of \eqref{dual}.
\end{definition}

\begin{definition}[Primal problem] We introduce a variational problem on the space of differentiable functions which is defined as follows
\thlabel{def: primal}
\begin{equation*}
\label{primal}
\inf_{\varphi \in C^1([0,1] \times \bar{\Om})} J(\varphi) \defeq \int_0^1 \int_{\Om} \iota_{B_{\delta}}(\D_t \varphi, \nabla \varphi, \varphi) + \int_{\Om} \varphi(0,\cdot)\d\rho_0 - \int_{\Om} \varphi(1,\cdot) \d\rho_1 \, .
\tag{$\mathcal{P}_{C^1}$}
\end{equation*}
\end{definition}
\begin{theorem}[Strong duality and existence of geodesics]
\thlabel{existence}
Let $\rho_0$ and $\rho_1$ in $\mathcal{M}_+(\Om)$. It holds
\[
\WF_{\delta}(\rho_0,\rho_1)^2 = -\inf_{\varphi \in C^1([0,1] \times \bar{\Om})} J(\varphi)
\] 
and the infimum in \eqref{dual} is attained.
\end{theorem}

\begin{proof}
First remark that $\WF_{\delta}(\rho_0,\rho_1)$ is necessarily finite from \thref{bounddistance}. Let us rewrite \eqref{primal} as 
\[
\inf_{\varphi \in C^1([0,1] \times \bar{\Om})} F(A\varphi) + G(\varphi)
\]
where
\begin{align*}
A : \varphi &\mapsto (\D_t \varphi, \nabla \varphi, \varphi)\, , \\
F :(\alpha,\beta,\gamma) &\mapsto \int_0^1 \int_{\Om} \iota_{B_{\delta}}(\alpha(t,x),\beta(t,x),\gamma(t,x)) \d x \d t \, , \\
 \text{and} \quad G : \varphi &\mapsto \int_{\Om} \varphi(0,\cdot)\d\rho_0 - \int_{\Om} \varphi(1,\cdot) \d\rho_1 .
\end{align*}
Remark that $F$ and $G$ are convex, proper and lower-semicontinuous. It is easy to find a function $\varphi \in C^1([0,1]\times \bar{\Om})$ such that $F$ is continuous at $A\varphi$, taking $A\varphi(t,x)$ in the interior of the closed set $B_{\delta}$ for all $(t,x)$ is enough. The Fenchel-Rockafellar duality theorem garantees the following equality
\begin{equation}
\inf_{\varphi \in C^1([0,1] \times \bar{\Om})} J(\varphi) = \max_{(\mu)\in \mathcal{M} \times \mathcal{M}^{d} \times \mathcal{M}} 
\left\{ -F^*(\mu)-G^*(-A^*(\mu)) \right\}.
\label{fenchelrocka}
\end{equation}
In the second term, we recognize the continuity constraint
\begin{align*}
G^*(-A^*\mu) &= \sup_{\varphi \in C^1([0,1]\times \Om)} \left\{ - \langle A\varphi, \mu \rangle    +\int_{\Om} \varphi(1,\cdot)\d\rho_1 - \int_{\Om} \varphi(0,\cdot) \d\rho_0 \right\}  \\
 &= 
\begin{cases}
0 & \text{ if $\mu \in \ccons$}\\
+\infty & \text{otherwise}.
\end{cases}
\end{align*}
For the first term, as the Fenchel conjugate of $\iota_{B_{\delta}}$ is $\fonc$, we can apply \cite[Theorem 5]{rockafellar1971integrals}. This theorem gives an integral representation for the conjugate of integral functionals by means of the recession function $f^{\infty}$, defined as $f^{\infty}(z) \defeq \sup_t f(x+tz)/t$, for $x$ such that $f(x)<+\infty$. More precisely,
\[
F^*(\mu) = \int_0^1 \int_{\Om} \fonc \left( \frac{\d\mu}{d\mathscr{L}} \right) \d \mathscr{L}+
\int_0^1 \int_{\Om} \fonc^{\infty} \left( \frac{\d\mu}{d|\mu^S|} \right) \d|\mu^S|
\]
where $\mathscr{L}$ is the Lebesgue measure, $\mu^S$ is any singular measure which dominates the singular part of $\mu$ and $\fonc^{\infty}$ is $\fonc$ itself by homogeneity. Finally we obtain $F^*(\mu)=\ifonc(\mu)$, which allows to recognize the opposite of \eqref{dual} in the right-hand side of \eqref{fenchelrocka}. The fact that the minimum is attained is part of the Fenchel-Rockafellar theorem.
\end{proof}

The following Theorem shows that we have defined a metric and shows two useful formulas. 

\begin{theorem}
\thlabel{WF defines a metric}
$\WF_{\delta}$ defines a metric on $\mathcal{M}_+(\Om)$. Moreover, we have the equivalent characterizations, by changing the time range
\begin{equation}
\label{alternative1}
\WF_{\delta}(\rho_0,\rho_1) = \left\{ \inf_{\mu \in \mathcal{CE}_0^{\tau}(\rho_0,\rho_1)} \tau  \int_{[0,\tau] \times \Om} \fonc\left(\frac{\d\mu}{\d\lambda}\right) \d\lambda \right\}^{\frac12} 
\end{equation}
with $\tau > 0$, and by disintegrating $\mu$ in time:
\begin{equation}
\label{alternative2}
\WF_{\delta}(\rho_0,\rho_1)= \inf_{\mu \in \mathcal{CE}_0^1(\rho_0,\rho_1)} \int_0^1 \left\{ \int_{\Om} \fonc\left( \frac{\d\mu_t}{\d\lambda_t}\right) \d\lambda_t \right\}^{\frac12} \d t
\end{equation}
where, as usual, we choose $(\lambda_t)_{t\in [0,1]}$ such that for all $t \in [0,1]$, $\mu_t \ll \lambda_t$.
\end{theorem}

\begin{remark}
The difference between minimizers of \eqref{dual} and \eqref{alternative2} is that for the former, the integrand is constant in time (as a consequence of  \eqref{alternative2}), i.e.\ minimizers of \eqref{dual} satisfy for almost every $t\in[0,1]$
\[
\WF_{\delta}(\rho_0,\rho_1) = \int_{\Om} \fonc\left( \frac{\d\mu_t}{\d\lambda_t} \right) \d\lambda_t
\]
while this is not necessarily the case for minimizers of \eqref{alternative2}.
\end{remark}

\begin{proof}
We obtain characterization \eqref{alternative1}, by remarking that if $T:t \mapsto \tau \cdot t$ is a time scaling and $\mu=(\rho,\M,\Z) \in \ccons$ then $\tilde{\mu}=(T_{\#} \rho, \frac{1}{\tau}T_{\#} \M, \frac{1}{\tau} T_{\#} \Z) \in \mathcal{CE}_0^{\tau} (\rho_0, \rho_1)$. For the proof of \eqref{alternative2}, the arguments  of \cite[Theorem 5.4]{dolbeault2009new} apply readily. Let us now prove that $\WF_{\delta}$ defines a metric. By definition, it is non-negative and it is symmetric because the functional satisfies $\ifonc(\rho,\M,\Z)=\ifonc(\rho,-\M,-\Z)$ so time can be reversed leaving $\ifonc$ unchanged. Finally, the triangle inequality comes from characterization \eqref{alternative2} and the fact that the continuity equation is stable by concatenation in time i.e.\ for $\tau \in ]0,1[$, if $\mu_1\in \mathcal{CE}_0^{\tau} (\rho_0, \rho_{\tau})$ and $\mu_2\in \mathcal{CE}_{\tau}^1 (\rho_{\tau}, \rho_1)$ then $\mu_1 1_{[0,\tau]}+\mu_2 1_{[\tau,1]} \in \ccons$. 
\end{proof}


\begin{proposition}[Covariance by mass rescaling]
\thlabel{mass rescaling}
Let $\alpha >0$. If $\rho$ is a geodesic for $\WF_{\delta}(\rho_0,\rho_1)$ then $\alpha \rho$ is a geodesic for $\WF_{\delta}(\alpha \rho_0,\alpha \rho_1)$ and 
\[
\WF_{\delta}(\alpha \rho_0,\alpha \rho_1)=\sqrt{\alpha} \WF_{\delta}(\rho_0,\rho_1)
\]
\end{proposition}

\begin{proof}
It is a direct consequence of the homogeneity of $\ifonc$.
\end{proof}

\subsection{Sufficient Optimality and Uniqueness Conditions}
\label{sec:optimality condition}

We now leverage tools from convex analysis in order to provide a useful condition ensuring uniqueness of a geodesic. 
We use this condition to study travelling Diracs in Section~\ref{sec-travelling-dirac}. 

\begin{lemma}
\thlabel{lemma:subdiff}
The subdifferential of $\ifonc$ at $\mu=(\rho,\M,\Z)$ such that $\ifonc (\mu) < + \infty$ contains the set
\begin{multline}
\label{subdiff}
\D^{c} \ifonc (\mu) = \Big\{ (\alpha,\beta,\gamma) \in C([0,1]\times \Om; B_{\delta}) : 
 \alpha+\frac12 \left(  |\beta|^2 + \frac{\gamma^2}{\delta^2}\right) = 0 - \rho \  a.e.,\  \\ 
  \beta \rho=\M \text{ and } \gamma \rho=\delta^2 \Z \Big\}.
\end{multline}
\end{lemma}

\begin{remark}
This set is exactly the set of continuous functions of the form $(\alpha, \beta,\gamma)$ which take their values in $\D \fonc$ (defined in \eqref{subdiff fonctional}) for all $(t,x)\in [0,1] \times \Om$. Note that $\D^{c} \ifonc$ can be empty.
\end{remark}

\begin{proof}
Take $\tilde{\mu}\in \mathcal{M}^{d+2}$ and choose $\rho^{\perp}\in \mathcal{M}_+$ such that $\rho$ and $\rho^{\perp}$ are mutually singular and $\lambda \defeq \rho + \rho^{\perp}$ satisfies $|\mu| + |\tilde{\mu}| \ll \lambda$. We have
\begin{align*}
\ifonc(\tilde{\mu})- \ifonc(\mu) 
&= \int_{[0,1] \times \Om} \left[\fonc \left( \frac{\d\tilde{\mu}}{\d\lambda} \right) - \fonc \left( \frac{\d\mu}{\d\lambda} \right) \right] \d\lambda \\
&= \int_{[0,1] \times \Om} \left[\fonc \left( \frac{\d\tilde{\mu}}{\d\lambda} \right) - \fonc \left( \frac{\d\mu}{\d\lambda} \right) \right] \d\rho + \int_{[0,1] \times \Om} \fonc \left( \frac{\d\tilde{\mu}}{\d\lambda} \right) \d\rho^{\perp} \\
& \overset{(*)}{\geq}  \int_{[0,1] \times \Om} \langle (\alpha,\beta,\gamma),  \frac{\d\tilde{\mu}}{\d\lambda}- \frac{\d\mu}{\d\lambda} \rangle  \d\rho+  \int_{[0,1] \times \Om} \langle (\alpha,\beta,\gamma) ,  \frac{\d\tilde{\mu}}{\d\lambda} \rangle \d\rho^{\perp}   \\
&= \langle (\alpha,\beta,\gamma), \tilde{\mu} - \mu \rangle
\end{align*}
and inequality $(*)$ holds for every $\tilde{\mu}$ if for all $(t,x)\in [0,1] \times \Om$, $(\alpha,\beta,\gamma)(t,x) \in \D \fonc \left(\frac{\d\mu}{\d\lambda}(t,x)\right)$: this comes from the definition of the subdifferential and from the fact the $\fonc = \iota_{B_{\delta}}^*$.
\end{proof}

\begin{theorem}[Sufficient optimality and uniqueness condition]
\thlabel{certificate}
Let $(\rho_0,\rho_1)\in \mathcal{M}^2_+(\Om)$. If $\mu=(\rho,\M,\Z) \in \ccons$ and there exists $\varphi \in C^1([0,T] \times \bar{\Om})$ such that
\[
(\D_t \varphi, \nabla \varphi, \varphi) \in \D^{c} \ifonc (\mu)  
\]
then $\rho$ is the \emph{unique} geodesic for $\WF_{\delta}(\rho_0,\rho_1)$. From now on, we will refer to such a $\varphi$ as an optimality certificate.
\end{theorem}

\begin{proof}
Consider another triplet $\tilde{\mu}\in \ccons$. We have, from the demonstration of the previous lemma 
\[
\ifonc(\tilde{\mu})- \ifonc(\mu) \geq \langle (\D_t \varphi, \nabla \varphi, \varphi), \tilde{\mu} - \mu \rangle = 0 \,
\]
showing that $\mu$ is a minimizer of $\ifonc$. For the proof of uniqueness, let us consider another minimizer $\tilde{\mu}=(\tilde{\rho},\tilde{\M}, \tilde{\Z}) \in \ccons$. It holds
\begin{align}
\ifonc(\tilde{\mu}) &
\geq \int_0^1 \int_{\Omega} \D_t \varphi \d \tilde{\rho} + \int_0^1 \int_{\Omega} \nabla \varphi \d \tilde{\M} + \int_0^1 \int_{\Omega}  \varphi \d \tilde{\Z} \\
&= \int_{\Omega} \varphi(1,\cdot)\d \rho_1 - \int_{\Omega} \varphi (0,\cdot) \d \rho_0 \\ 
&= \ifonc(\tilde{\mu})
\end{align}
where we used successively: duality, the fact that $\tilde{\mu} \in \ccons$ and the fact that the primal-dual gap vanishes at optimality. So, the first equality is an equality, and hence $\lambda$ a.e., $(\D_t \varphi, \nabla \varphi, \varphi)(t,x) \in \D \fonc \left( \frac{\d \tilde{\mu}}{\d \lambda} (t,x) \right)$ for $\lambda$ such that $\tilde{\mu}\ll \lambda$. And thus, from the characterization of $\D \fonc$, we have 
\[
\tilde{\mu} = (\tilde{\rho}, (\nabla \varphi) \tilde{\rho}, \varphi \tilde{\rho}).
\]
But $\tilde{\mu}\in \ccons$ and thus $\sigma = (\rho^{\sigma}, \M^{\sigma}, \Z^{\sigma}) \defeq \tilde{\mu}-\mu \in \mathcal{CE}_0^1(0,0)$. 
Recall (see the remarks after \thref{continuity equation}) that the map $t \mapsto \rho^{\sigma}_t(\Om)$ admits the distributional derivative $ \Z^{\sigma}_t(\Om)$, for almost every $t\in[0,1]$. As a consequence,
\[
\rho^{\sigma}(\Om)'(t) \leq \|\varphi\|_{\infty} \rho^{\sigma}(\Om)(t)
\]
since $\varphi$ is continuous on the compact domain $\bar{\Om}$.
But $\rho^{\sigma}(\Om)(0)=0$ thus $\rho^{\sigma}= 0$ and $\tilde{\rho}=\rho$. Consequently $\tilde{\mu}=\mu$.

 \end{proof}

\begin{remark}[Relaxation]
The above condition is not necessary: the optimality certificate $\varphi$ is sometimes not as smooth as required by this condition. It is an interesting problem to study the space in which the infimum of \eqref{primal} is attained, as it would give a necessary optimality condition. For instance, in the case of dynamical optimal transport, this space is that of Lipschitz functions \cite{jimenez2008dynamic}. In a different setting, for a class of distance which interpolates between optimal transport and Sobolev, \cite{Cardaliaguet:2012aa} proves that the optimum is attained in $BV \cap L^{\infty}$, where $BV$ is the set of functions with bounded variation. This question is non-trivial and beyond the scope of the present paper.
\end{remark}



\section{Interpolation Properties}
\label{sec-interpolation}

We have considered $\delta$ as a constant so far, in this section we study the effect of varying $\delta$ for a given problem. This allows us to explain to which extent our metric really interpolates between optimal transport and Fisher-Rao.

\subsection{Change of Scale}

Intuitively, the smaller $\delta$, the finer the scale at which mass creation and removal intervene to make initial and final measures match. Accordingly, we show in the following proposition that changing $\delta$ amounts to changing the scale of the problem.

\begin{proposition}[Space rescaling]
\thlabel{rescaling}
Let $T: (t,x) \in \R \times \R^d \mapsto (t,s\, x)$ be the spatial scaling by a factor $s>0$, and $\delta \in ]0,+\infty[$. If $\mu=(\rho,\M,\Z)$ is a minimizing triplet for the distance $\WF_{\delta}(\rho_0,\rho_1)$, then $\hat{\mu}=(T_{\#}\rho,s\, T_{\#}\M,T_{\#}\Z)$ is a minimizing triplet for the distance $\WF_{s\, \delta}(T_{\#}\rho_0,T_{\#}\rho_1)$.
\end{proposition}

\begin{proof}
It is easy to check that $\hat{\mu}=(\hat{\rho},\hat{\M},\hat{\Z})$ satisfies the continuity equation with Neumann boundary conditions on $T(\Om)$ between $T_{\#}\rho_A$ and $T_{\#}\rho_B$. Also, we can write
\[
\ifoncN_{s\delta}(\hat{\mu}) 
= \int_{T([0,1]\times \Om)} \foncN_{s\delta} \left( \frac{d\hat{\mu}}{d\hat{\lambda}} \right) \d\hat{\lambda}
\]
where $\hat{\lambda}=T_{\#} \lambda$ and $\lambda$ is any non-negative Borel measure satisfying $|\mu| \ll \lambda$. By denoting $ \frac{\d\mu}{\d\lambda} =(\rho^{\lambda},\M^{\lambda},\Z^{\lambda})$ we have $\frac{d\hat{\mu}}{d\hat{\lambda}}(t,x)=(\rho^{\lambda}(t,x/s),s\M^{\lambda}(t,x/s),\Z^{\lambda}(t,x/s))$. Thus by the change of variables formula, we obtain
\begin{eqnarray*}
\ifoncN_{s\delta}(\hat{\mu}) 
&=& \int_{[0,1]\times \Om} \foncN_{s\delta}\left(\rho^{\lambda},s\, \M^{\lambda},\Z^{\lambda} \right) \d\lambda\\
&=& s^2 \ifoncN_{\delta}(\mu) \, .
\end{eqnarray*}
This shows that if $\mu$ minimizes $\ifonc$ then $\hat{\mu}$ minimizes $D_{s\delta}$ and reciprocally, hence \thref{rescaling}.
\end{proof}

\begin{remark}
The proof of \thref{rescaling} holds true as soon as $f_{\delta}$ is of the form $f_1(\rho,\M,\Z)+ \delta^p f_2(\rho,\Z)$ where $f_1$ is p-homogeneous in $\M$ and $f_2$ is independent of $\M$. Hence the choice of $\delta^p$ as an interpolation factor when introducing the various models in Section \ref{previous works}: so that $\delta$ can be interpreted as a spatial scale.
\end{remark}

The following result is closely related and will be used later to prove uniqueness of geodesics between remote Diracs.
\begin{proposition}[Monotonicity of the distance w.r.t.\ rescaling]
\thlabel{monotonicity of the distance}
Using the same notations as in the previous result, if $s<1$, then $D_{\delta}(\hat{\mu}) \leq D_{\delta}(\mu)$ and thus $\WF_{\delta}(T_{\#}\rho_0,T_{\#}\rho_1) \leq \WF_{\delta}(\rho_0, \rho_1)$. Moreover, if $s<1$ and $\omega \neq 0$ these inequalities are strict.
\end{proposition}

\begin{proof}
We use the same change of variables as in the proof of \thref{rescaling} and remark that $\fonc$ is strictly increasing w.r.t.\ its second variable.
\end{proof}

\subsection{Two variational problems on non-negative measures}

The main Theorem of section \ref{sec:limit metrics} states that the geodesics of the quadratic Wasserstein and the Fisher-Rao metrics are both recovered (in a suitable sense) when letting the parameter $\delta$ go to $+\infty$ and $0$, respectively. We start by defining two variational problems, before showing their connections to $\WF_{\delta}$ in the next subsection.

Let us introduce
\[
D_{BB} : (\rho,\M) \mapsto \ifonc(\rho,\M,0) \quad \text{and} \quad D_{FR} : (\rho,\Z) \mapsto \frac{1}{\delta^2}\ifonc(\rho,0,\Z).
\]
Those real valued functionals represent respectively the Benamou-Brenier and the Fisher-Rao terms in the interpolating functional $\ifonc$. They do not depend on $\delta$ and are defined such that $\ifonc (\rho,\M,\Z)= D_{BB}(\rho,\M)+ \delta^2 D_{FR}(\rho,\Z)$.

\begin{definition}[Generalized Benamou-Brenier transport problem]
\thlabel{def gBB}
Given $(\rho_0,\rho_1)$ $\in \mathcal{M}_+(\Om)$, we introduce the time dependent \emph{rate of growth}
\begin{equation}
\label{rate of growth}
g: (t,x) \mapsto 
\begin{cases}
0 &\text{ if $\rho_0(\Om)=\rho_1(\Om)$} \\
\frac{2}{t-t_0} & \text{otherwise}
\end{cases}
\end{equation}
with $t_0=\sqrt{\rho_0(\Om)}\left( \sqrt{\rho_0(\Om)}-\sqrt{\rho_1(\Om)}\right)^{-1}$. The generalized Benamou-Brenier transport problem is defined as
\begin{eqnarray} %
\label{eq gBB}
& \underset{\rho,\M}{\inf} &D_{BB}(\rho,\M) \\
& \text{subject to}  & (\rho,\M,g \rho) \in \ccons\, . 
\end{eqnarray}
and we denote by $d_{gBB}(\rho_0,\rho_1)$ the square-root of the infimum.
\end{definition}

Notice that $d_{gBB}$ belongs to the class of models for unbalanced transport of \cite{lombardi2013eulerian} which prescribe the source term $\Z$ as a function of $\rho$. In order to picture what is happening, consider $(\rho_0,\rho_1)\in \mathcal{M}_+(\Om)$ such that $\rho_0(\Om) < \rho_1(\Om)$: in that case, $t_0<0$ and the rate of growth $g$ is positive, constant in space and decreasing in time.
While $d_{gBB}$ does not define a metric as it does not satisfy the triangle inequality, we prove below that it is closely related to the quadratic Wasserstein metric.
\begin{proposition}[Relation between $d_{gBB}$ and $W_2$]
\thlabel{th : link gBB}
Let $\rho_0,\, \rho_1 \in \mathcal{M}_+(\Om)$. We introduce $\tilde{\rho}_0, \, \tilde{\rho}_1$ the rescaled measures which mass is the geometric mean between $\rho_0(\Om)$ and $\rho_1(\Om)$, i.e.\ such that for $i\in \{0,1\}$, $\tilde{\rho}_i(\Om)=\sqrt{\rho_0(\Om)\rho_1(\Om)}$ and $\tilde{\rho_i} = \alpha_i \rho_i$ for some $\alpha_i \geq 0$. It holds
\[
d_{gBB}(\rho_0,\rho_1) = \frac12 W_2(\tilde{\rho}_0,\tilde{\rho}_1) \, ,
\]
Moreover, the minimum in \eqref{eq gBB} is attained and minimizers can be built explicitly from the geodesics of $W_2(\tilde{\rho}_0,\tilde{\rho}_1)$.
\end{proposition}

\begin{proof}
If $\rho_0=0$ or $\rho_1=0$, the minimum is attained if and only if $\M=0$ and thus the $d_{gBB}(\rho_0,\rho_1)=0$. If $\rho_0(\Om)=\rho_1(\Om)$ then $g = 0$ and the conclusion follows directly.
Otherwise, we have $t_0 \notin [0,1]$ and we will bring ourselves back to the case of a standard Benamou-Brenier problem by a rescaling and a change of variables in time, similarly as in \cite[Proposition 7]{lombardi2013eulerian}. Consider 
\begin{align*}
\mathcal{S} &= \left\{ (\rho,\M) : (\rho,\M,g\rho) \in \mathcal{CE}_0^1(\rho_0,\rho_1), \, |\M| \ll \rho \right\} \\
\text{and} \quad
\mathcal{S}_0 &= \big\{ (\rho,\M) : (\rho, \M, 0) \in  \mathcal{CE}_0^1(\rho_0,\frac{\rho_0(\Om)}{\rho_1(\Om)}\rho_1),  \, |\M| \ll \rho  \big\}\, .
\end{align*}
Take $\mu=(\rho,\M)$ in $\mathcal{S}$. It satisfies (weakly) the ordinary differential equation $\D_t \rho_t(\Om) = g(t) \rho_t(\Om)$ and so $\rho_t(\Om) = \left( \frac{t_0-t}{t_0}\right)^2 \rho_0(\Om)$ a.e.\ in $[0,1]$. We define the rescaling
\begin{align*}& R : \mathcal{S} \mapsto  \mathcal{S}_0 \\
& R(\rho,\M) =  \left( \frac{t_0}{t_0-t}\right)^2 (\rho,\M) \, .
\end{align*}
which is a bijection as $t_0 \notin [0,1]$.
Writing $D_{BB}(R(\rho,\M))$ yields a Benamou-Brenier functional with a time varying metric;  we now counterbalance the latter by a change of variables in time.
Consider the strictly increasing map $\s: t \mapsto t\frac{t_0-1}{t_0-t}$ which satisfies $\s(0)=0$, $s(1)=1$ and $\s'(t)=\frac{t_0-1}{t_0}\left( \frac{t_0}{t_0-t} \right)^2>0$. Let $\t = \s^{-1}$ and introduce
\begin{align*}& T: \mathcal{S}_0 \mapsto  \mathcal{S}_0 \\
&T(\rho,\M) = (\rho \circ \t(s), \t'(s) \cdot \M \circ \t(s))\,
\end{align*}
which is a bijection (see \cite[Lemma 8.1.3]{ambrosio2006gradient}).
The rescaling followed by the change of variables in time induces a bijection $T\circ R : \mathcal{S} \to \mathcal{S}_0$. Of course, the expression for $\s$ has been chosen after an analysis of its effect on the $D_{BB}$ functional. 
Take $(\tilde{\rho},\tilde{\omega}) = R(\rho,\omega) \in \mathcal{S}_0 $ and $(\hat{\rho},\hat{\omega}) = T \circ R(\rho,\omega) \in \mathcal{S}_0 $. All those measures can be disintegrated in time from the definition of $\mathcal{S}$ and $\mathcal{S}_0$. In order to alleviate notations, we introduce $f_{BB}(\rho_t,\M_t)=\int_{\Omega} \fonc \left(\frac{\d \rho}{\d \lambda}, \frac{\d \M_t}{\d \lambda_t},0 \right) \d \lambda_t $,  where $\lambda_t \in \mathcal{M}_+(\Omega)$ satisfies $|(\rho_t, \M_t)|\ll \lambda_t$. We have
\begin{align*}
\int_0^1 f_{BB}(\hat{\rho}_s, \hat{\M}_s) \d s
&= \int_0^1 (\t' \circ \s (t))^2 f_{BB}(\tilde{\rho}_t, \tilde{\M}_t) \s'(t) \d t \\
&= \int_0^1 \frac{f_{BB}(\tilde{\rho}_t, \tilde{\M}_t)}{\s'(t)} \d t \\
&= \frac{t_0}{t_0-1} \int_0^1 \left(\frac{t_0-t}{t_0} \right)^2 f_{BB}(\tilde{\rho}_t, \tilde{\M}_t) \d t \\
&= \left(\frac{\rho_0(\Om)}{\rho_1(\Om)}\right)^{1/2} \int_0^1 f_{BB}(\rho,\M) \d t
\end{align*}
Hence the relation
\[
D_{BB}(\rho,\M) = \left(\frac{\rho_1(\Om)}{\rho_0(\Om)}\right)^{1/2} D_{BB}(T \circ R (\rho,\M))   \, ,
\]
from which we deduce $d_{gBB}(\rho_0,\rho_1) = \frac12 \left(\frac{\rho_1(\Om)}{\rho_0(\Om)}\right)^{1/4} W_2(\rho_0,\frac{\rho_0(\Om)}{\rho_1(\Om)} \rho_1)$, as the minimum of the standard Benamou-Brenier problem is attained in $\mathcal{S}_0$. Finally, by remarking that $W_2(\alpha \rho_0, \alpha \rho_1) = \sqrt{\alpha}W_2(\rho_0,\rho_1)$ for $\alpha\geq0$, we obtain the formulation of \thref{th : link gBB}.
\end{proof}

\begin{definition}[Fisher-Rao metric]
\thlabel{def FR}
Given $(\rho_0,\rho_1) \in \mathcal{M}_+(\Om)$, the Fisher-Rao distance is defined as
\begin{eqnarray*}
d_{FR}^2(\rho_0,\rho_1) \defeq & \underset{\rho,\Z}{\inf} & D_{FR} (\rho,\Z) \\
&\text{subject to} & (\rho,0,\Z) \in \ccons \, .
\end{eqnarray*}
\end{definition}

We now prove uniqueness of Fisher-Rao geodesics and give their explicit expression. This result is well known in a positive, smooth setting.
\begin{theorem}
\thlabel{uniqueness FR}
The geodesics for the Fisher-Rao metric are unique and have the explicit expression 
\begin{equation}
\label{FR geodesic}
\rho = (t\sqrt{\rho_1} + (1-t)\sqrt{\rho_0})^2 \otimes  \d t\, .
\end{equation}
\end{theorem}

\begin{proof}
First notice that the expression \eqref{FR geodesic} is not ambiguous since it is positively homogeneous as a function of $(\rho_0,\rho_1)$. The fact that $d_{FR}$ defines a metric is proven from the same arguments as in \thref{WF defines a metric}. As for the proof of uniqueness of geodesics, we organize the proof as follows. Posing $\nu \defeq \rho_0 + \rho_1$,  we will first show that the geodesics are absolutely continuous w.r.t.\ $\nu \otimes \d t$, allowing to restrict the problem to trajectories which time marginals are $L^1(\d \nu)$. Using convex duality, we then show that \eqref{FR geodesic} is a geodesic, and this allow us to exhibit an isometric injection (the square root) into $L^2(\d\nu)$, from which we deduce uniqueness. 

\paragraph{i. Duality.}
Written explicitly, the constraint in \thref{def FR} is equivalent to
\begin{equation}
\label{continuity FR}
\int_0^1 \int_{\Omega} \partial_t \varphi \d \rho + \int_0^1 \int_{\Omega} \varphi \d \Z = \int_{\Omega}  \varphi(1,\cdot) \d \rho_1 - \int_{\Omega} \varphi(0,\cdot)\d  \rho_0 \, 
\end{equation}
for all test functions $\varphi \in C^1( [0,1] \times \Omega)$.
Similarly as in \thref{existence}, we obtain that the problem defining the Fisher-Rao metric is the dual of
\[
\sup_{\varphi \in C^1( [0,1] \times \Omega)} \int_{\Omega} \varphi(1,\cdot) \d\rho_1 - \int_{\Omega} \varphi(0,\cdot) \d\rho_0 
\]
subject to, for all $(t,x)\in [0,1] \times \Omega$,
\[
 (\partial_t \varphi(t,x), \varphi(t,x))\in B_{FR}  \quad \text{where} \quad  B_{FR} = \left\{ (a,b) : a+\frac12 b^2 \leq 0 \right\}  \, .
\]
By the Fenchel-Rockafellar theorem, we obtain the existence of minimizers and standard arguments show that geodesics are constant speed and that $d_{FR}$ defines a metric on $\mathcal{M}_+$.

\paragraph{ii. Stability of geodesics.}
Take $\mu=(\rho,\Z)$ a minimizer for $d_{FR}(\rho_0,\rho_1)$. First of all, we can always construct paths of finite energy (take \eqref{FR geodesic} for instance), so it holds $D_{FR}(\mu)< + \infty$ and thus $\Z \ll \rho$. As a consequence, we can disintegrate $\mu$ in time. Now decompose $\mu_t$ w.r.t.\ $\nu \defeq \rho_0 + \rho_1$ as $\mu_t = (r(t,\cdot), z(t,\cdot)) \nu + (\rho^{\perp}, \Z^{\perp})$. We remark that $(r(t,\cdot), z(t,\cdot)) \nu$ satisfies the constraints \eqref{continuity FR} and is thus a minimizer. Thus, $\Z^{\perp} = 0$ because $\mu$ is also a minimizer and consequently, $\rho^{\perp} = 0$ by condition \eqref{continuity FR}. Thus, $\mu= (r,z) \nu \otimes \d t$, with $r(t,\cdot), \, z(t,\cdot) \in L^1(\d \nu)$ for all $t\in [0,1]$.

\paragraph{iii. Value of $d_{FR}$.} For the moment, assume that $\rho_0 = r_0 \nu$, $\rho_1 = r_1 \nu$ and that there exists $\kappa>0$ such that $\nu$ a.e.,
\[
\frac{1}{\kappa} \leq \frac{r_0}{r_1} \leq \kappa\, . \tag{A1}
\]
 Consider $\mu=(\rho,\Z)$ where $\rho$ is defined as in \eqref{FR geodesic}, i.e.
 \begin{align*}
 \rho &= (t\sqrt{r_1} +(1-t)\sqrt{r_0})^2 \nu \otimes \d t\\
 \Z    &= 2(\sqrt{r_1}-\sqrt{r_0})(t\sqrt{r_1}-(1-t) \sqrt{r_0}) \nu \otimes \d t\, .
 \end{align*}
 This couple satisfies the constraints \eqref{continuity FR} and $D_{FR}(\mu)=2\int_{\Omega}(\sqrt{\rho_1}-\sqrt{\rho_0})^2$. 
Now consider 
\[
\varphi^* \defeq \zeta/\rho= 
\begin{cases}
0 & \text{if $r_0(x) = r_1(x)$,} \\
\frac{2}{t-t_0(x)} & \text{otherwise.}
\end{cases}
\] with $t_0(x)=\sqrt{r_0}\left( \sqrt{r_0}-\sqrt{r_1} \right)^{-1}$. Under the assumption (A1), there exists $\epsilon>0$ such that for all $x\in \Omega$, $t_0(x) \notin [-\epsilon,1+\epsilon]$ and thus, for all $t\in[0,1]$, $\varphi^*(t,\cdot)$ is bounded and belongs to $L^1(\d \nu)$. Also, for all $x\in \Omega$, $\varphi^*(\cdot,x)\in C^1([0,1])$. But, in general, $\varphi^* \notin C^1( [0,1] \times \Omega)$ as it lacks regularity w.r.t.\ the space variable.

We introduce thus its regularized version: $\varphi^{\epsilon} \defeq \eta^{\epsilon} \ast \varphi^*$, where $\eta^{\epsilon}(t,x)\defeq\epsilon^{-d}\alpha(x/\epsilon)$ with $\alpha \in C_c^{\infty}((-1/2,1/2)^d)$, $\alpha \geq 0$, $\int \alpha = 1$ and $\alpha$ even.
By convexity of $B_{FR}$, we have for all $(t,x)\in [0,1] \times \Om$, 
$(\partial_t \varphi^*, \varphi^*)(t,x) \in B_{FR} \Rightarrow (\eta^{\epsilon} \ast \partial_t \varphi^*, \eta^{\epsilon} \ast \varphi^*)(t,x) \in B_{FR} \Rightarrow (\partial_t \varphi^{\epsilon}, \varphi^{\epsilon})(t,x)   \in B_{FR}$.
Thus
\begin{align*}
d^2_{FR}(\rho_0,\rho_1) 
&\geq \lim_{\epsilon \rightarrow 0} \int_{\Omega} \varphi^{\epsilon}(1,\cdot) \d \rho_1 - \int_{\Omega} \varphi^{\epsilon}(0,\cdot) \d \rho_0 \\
&= \int_{\Omega} \varphi^*(1,\cdot) \d \rho_1 - \int_{\Omega} \varphi^*(0,\cdot) \d \rho_0 \\
&= 2\int_{\Omega} (\sqrt{\rho_1}-\sqrt{\rho_0})^2\, .
\end{align*}
As this value is also an upper bound, this shows that, under (A1), $d^2_{FR}(\rho_0,\rho_1) = 2\int_{\Omega} (\sqrt{\rho_1}-\sqrt{\rho_0})^2$.
But this result remains true without any assumption on $\rho_0$ and $\rho_1$. Indeed, by introducing $\rho^{\epsilon}_1=\epsilon \rho_0 + (1-\epsilon) \rho_1$ and $\rho^{\epsilon}_0=\epsilon \rho_1 + (1-\epsilon) \rho_0$, the triangle inequality yields
\[
d_{FR}(\rho^{\epsilon}_0,\rho^{\epsilon}_1) - d_{FR}(\rho^{\epsilon}_1,\rho_{1}) \leq
d_{FR}(\rho^{\epsilon}_0,\rho_{1}) \leq
d_{FR}(\rho^{\epsilon}_0,\rho^{\epsilon}_1) + d_{FR}(\rho^{\epsilon}_1,\rho_{1}) 
\]
and $\lim_{\epsilon \rightarrow 0} d_{FR}(\rho^{\epsilon}_1,\rho_{1} ) = 0$ because we have a vanishing upper bound. Hence $d^2_{FR}(\rho^{\epsilon}_0,\rho_{1}) = 2\int_{\Omega} (\sqrt{\rho_1}-\sqrt{\rho_0^{\epsilon}})^2$.  Repeating this operation for $\rho_0^{\epsilon}\rightarrow \rho_0$ we have that $d^2_{FR}(\rho_0,\rho_1) = 2\int_{\Omega} (\sqrt{\rho_1}-\sqrt{\rho_0})^2$ for $\rho_0, \rho_1 \in \mathcal{M}_+(\Omega)$.

\paragraph{iv. Isometric injection.} Finally, on the space $\mathcal{M}_{\nu} = \{ \rho \in \mathcal{M}_+(\Omega) : \rho \ll \nu\}$, the map  
\[
\left|
\begin{array}{ccc}
(\mathcal{M}_{\nu}, d_{FR}) & \rightarrow & (L^2(\d \nu), \Vert \cdot \Vert_2) \\
\rho = \alpha\, \nu & \mapsto & \sqrt{2\alpha}
\end{array}
\right.
\]
is an isometric injection. As geodesics in $L^2(\d \nu)$ are unique, uniqueness holds for geodesics in $\mathcal{M}_{\nu}$  and thus for geodesics in $\mathcal{M}_+(\Omega)$.
\end{proof}
The following lemma motivates the expressions for $d_{gBB}$ and $d_{FR}$.

\begin{lemma}[Alternative characterizations]
\thlabel{origin limit metrics}
For all $(\rho_0,\rho_1)\in \mathcal{M}_+(\Om)$ we have
\begin{align*}
d_{gBB}(\rho_0,\rho_1)^2 &=  \min_{(\rho,\M)\in \mathcal{A}_{FR}}  D_{BB}(\rho,\M) \\
d_{FR}  (\rho_0,\rho_1)^2 &=  \min_{(\rho,\Z)\in \mathcal{A}_{BB}}  D_{FR}(\rho,\Z) 
\end{align*}
where 
\begin{equation}
\label{argmin sets}
\mathcal{A}_{FR} \defeq \argmin_{(\rho,\M,\Z) \in \ccons} D_{FR} (\rho,\Z)
\quad \text{and} \quad
\mathcal{A}_{BB} \defeq \argmin_{(\rho,\M,\Z) \in \ccons} D_{BB} (\rho,\M)\,.
\end{equation}
\end{lemma}

\begin{proof}
What we need to show is
\begin{align*}
\mathcal{A}_{BB} &= \left\{ (\rho,\Z) : (\rho,0,\Z) \in \ccons \right\} \\
\mathcal{A}_{FR} &= \left\{ (\rho,\M) : (\rho,\M,g \rho) \in \ccons  \right\}\, ,
\end{align*}
where the \emph{rate of growth} $g$ is defined in \eqref{rate of growth}. The first equality is easy because $D_{BB}(\rho, \omega)=0$ if and only if $\omega=0$ a.e.

The second equality requires slightly more work. In the case when $\rho_0(\Om)=\rho_1(\Om)$, $D_{FR}$ is minimized if and only if $\Z=0$ as previously. The case $\rho_0(\Om)=0$ or $\rho_1(\Om)=0$ is dealt with in in \thref{rescaled measures} (case $\alpha=0$). Otherwise, let us determine $\mathcal{A}_{FR}$ by adapting the optimality condition in \thref{certificate}. The argument is more clear if we consider $D_{FR}$ as a function of  $\mu=(\rho,\M,\Z)$ instead of just $(\rho,\Z)$ as $\M$ is a variable of the problem. We thus introduce
\[
B_{FR} \defeq \left\{ (\alpha,\beta,\gamma)\in \R \times \R^d \times \R : \alpha +\frac{\gamma^2}{2} \leq 0, \; \beta =0 \right\}
\]
and it can be shown (similarly as in Section \ref{sec:optimality condition}) that $\D D_{FR} (\rho,\M,\Z)$ is equal to
\[
 \left\{ (\alpha,\beta, \gamma) \in C([0,T] \times \bar{\Om}; B_{FR}) :
\alpha+\frac{\gamma^2}{2} = 0 -\text{ $\rho$ a.e.\ and $\gamma \rho = \Z$ }
\right\}.
\]
Adapting the sufficient optimality condition given in \thref{certificate}, it holds that if $\mu\in \ccons$ and if
there exists $\varphi \in C^1([0,1] \times \bar{\Om})$ satisfying $(\D_t \varphi, \nabla \varphi, \varphi)\in \D D_{FR} (\mu) $ i.e.
\[
\begin{cases}
 \D_t \varphi + \frac12 \varphi^2 \leq 0 &(\text{with equality $\rho$ a.e.}),\\
\nabla \varphi = 0\, ,\\
\varphi \rho = \Z \, .
\end{cases}
\] 
then $\mu \in \mathcal{A}_{FR}$.
The equality $\rho$ a.e.\ is actually an equality $\d t$ a.e.\ in time as $\rho$ has a positive mass and $\varphi$ is constant in space. All solutions are thus of the form $\varphi : (t,x) \mapsto \frac{2}{t-t_0}$. The integration constant $t_0$ is found by integrating the variations of mass. As $t_0\notin [0,1]$, $\varphi$ is well defined and is equal to $g$ as introduced in \eqref{rate of growth}. Now we prove that this condition is actually necessary: take $(\rho,\M,\Z) \in \mathcal{A}_{FR}$. It holds
\[
D_{FR}(\rho,\M,\Z) \geq \int_0^1 \int_{\Omega} \D_t g \rho + g \Z = \int_{\Omega} g(1)\rho_1 - \int_{\Omega} g(0) \rho_0 = D_{FR}(\rho,\M,\Z)
\]
where we used successively: duality, the fact that $(\rho,\M,\Z)\in \ccons$ and the fact that the primal-dual gap vanishes at optimality. So, the first equality is an equality, and hence $(\D_t g,g)(t,x) \in \D f_1(\d \rho / \d |\rho| , 0, \d \Z / \d | \rho|)(t,x)$ for $|\rho|$ a.e.\ $(t,x)\in [0,1]\times \Om$, which implies that $\Z = g\rho$.
\end{proof}

\subsection{Convergence of geodesics}
\label{sec:limit metrics}

A rigorous treatment of the \emph{limit geodesics} requires some compactness for the set of geodesics when $\delta$ varies. That is why we need the following bounds, adapted from \cite[Proposition 3.10]{dolbeault2009new}.

\begin{lemma}
\thlabel{estimates}
Consider a triplet $\mu=(\rho,\M,\Z)\in \mathcal{M}^{d+2}$ such that $D_{BB}(\rho,\M) <+\infty$ and $D_{FR}(\rho,\Z)<+\infty$. For any non-negative Borel function $\eta$ on $[0,1] \times \bar{\Om}$ we have
\begin{equation}
\int_{[0,1]\times\Om} \eta d|\M| \leq \sqrt2 D_{BB}(\rho,\M)^{\frac12} \left( \int_{[0,1]\times\Om} \eta^2 \d\rho \right)^{\frac12}
\label{estimate1}
\end{equation}
and
\begin{equation}
\int_{[0,1]\times\Om} \eta d|\Z| \leq \sqrt2  D_{FR}(\rho,\Z)^{\frac12} \left( \int_{[0,1]\times\Om} \eta^2 \d\rho \right)^{\frac12}
\label{estimate2}
\end{equation}
Also, similar bounds can be written for $\mu_t$ (the disintegration of $\mu$ w.r.t.\ time) by integrating solely in space.
\end{lemma}

\begin{proof}
Take $\lambda\in \mathcal{M}_+$ such that $|\mu| \ll \lambda$ and decompose $\d\mu$ as $(\rho^{\lambda}, \M^{\lambda}, \Z^{\lambda}) \d\lambda$. As $D_{BB}(\rho,\M)<+\infty$ and $D_{FR}(\rho,\Z)<+\infty$, we have $\M \ll \rho$ and $\Z \ll \rho$. It follows
\begin{align*}
\int_{[0,1]\times \Om} \eta \d |\M| 
&=  \int_{[0,1]\times \Om} \left( 2 \eta^2 \fonc(\rho^{\lambda},\M^{\lambda},0)\rho^{\lambda} \right)^{\frac12} \d\lambda \\
&\leq \sqrt2 \left( \int_{[0,1]\times \Om} \eta^2 \rho^{\lambda} \d\lambda \right)^{\frac12} 
\left( \int_{[0,1]\times \Om} \fonc(\rho^{\lambda}, \M^{\lambda},0) \d\lambda \right)^{\frac12}
\end{align*}
by Cauchy-Schwartz inequality on the scalar product $\langle \cdot, \cdot \rangle_{L^2(\d\lambda)}$. Inequality \eqref{estimate1} follows and \eqref{estimate2} is derived similarly.
\end{proof}

In the following lemma, we derive a total variation bound on $\mu$ which only depends on $D_{BB}(\mu)$ and $D_{FR}(\mu)$.

\begin{lemma}[Uniform bound]
\thlabel{upperbounds}
Let $\rho_0, \rho_1 \in \mathcal{M}_+(\Om)$ and $M\in \R_+$. There exists $C\in \R_+$ satisfying  $|\mu|([0,1]\times \Om) < C$, for all $\mu=(\rho,\M,\Z)\in \ccons$ such that $D_{BB}(\rho,\M) <M$ and $D_{FR}(\rho,\Z) <M$.
\end{lemma}

\begin{proof}
We start by giving a bound on $\rho([0,1] \times \Om)=\int_0^1 \rho_t(\Om) \d t$. Recall (from the remarks after \thref{continuity equation}) that the map $t \mapsto \rho_t(\Om)$ admits the distributional derivative $\rho'_t (\Om) = \zeta_t(\Om)$, for almost every $t\in[0,1]$. It follows, by \thref{estimates}
\[
\rho'_t (\Om) \leq |\zeta_t \vert(\Om) \leq \sqrt{2 D_{FR}(\rho_t,\Z_t) \rho_t(\Om)}
\]
where $D_{FR}(\mu_t)$ denotes $\frac{1}{\delta^2} \int_{\Om} \fonc(\frac{d(\rho_t,0,\Z_t)}{\d\lambda_t})\d\lambda_t$ where $\lambda_t$ is such that $\mu_t \ll \lambda_t$. By integrating in time and applying the Cauchy-Schwartz inequality
\[
\rho_t(\Om)-\rho_0(\Om) 
\leq \sqrt{2D_{FR}(\rho,\Z)} \sqrt{ \rho([0,1] \times \Om)} .
\]
We integrate again in time to obtain 
\[  
	\rho([0,1] \times \Om) \leq \rho_0 (\Om) + \sqrt{2D_{FR}(\rho,\Z)} \sqrt{ \rho([0,1] \times \Om)} 
\]
which implies that $ \rho([0,1] \times \Om)$ is bounded by a constant depending only on $D_{FR}(\rho,\Z)$ and $\rho_0(\Om)$. 
The conclusion follows by bounding  $|\Z|([0,1]\times \Om)$ and $|\M|([0,1]\times \Om)$ thanks to \thref{estimates}.
\end{proof}

Let us now recall a well known property of weighted optimization problems.

\begin{lemma}[Properties of weighted optimization problems]
\thlabel{weighted optim}
Let $f$ and $g$ be two proper l.s.c.\ functions on a compact set $\mathcal{C}$ with values in $\R\cup \{+\infty \}$ such that $h = f + \delta g$ admits a minimum in $\mathcal{C}$ for all $\delta \in ]0,±\infty[$. Let $(\delta_n)_{n\in\N}$ be a sequence in $]0,+\infty[$ and $(x_n)_{n\in \N}$ a sequence of minimizers of $h_n = f+\delta_n g$. We introduce $\mathcal{A}(f) \defeq \argmin_{x\in \mathcal{C}} f(x)$ and $\mathcal{A}(g)$ similarly.
\begin{enumerate}
\item For all $n\in \mathbb{N}$, $x^f\in \mathcal{A}(f)$ and $x^g\in \mathcal{A}(g)$ we have $f(x_n)\leq f(x^g)$ and $g(x_n) \leq g(x^f)$,
\item If  $\delta_n \underset{n \to \infty}{\to} 0$, then $(x_n)_{n\in \N}$ admits an accumulation point in $ \argmin_{\mathcal{A}(f)} g $,
\item If  $\delta_n \underset{n \to \infty}{\to} \infty$, then $(x_n)_{n\in \N}$ admits an accumulation point in $\argmin_{\mathcal{A}(g)} f$.
\end{enumerate}
\end{lemma}

\begin{proof}
All those results are derived using elementary inequality manipulations.
\end{proof}
The following lemma gives total variation bounds on geodesics that are independent of $\delta$, which will be helpful in order to explicit the limit models.

\begin{lemma}[A relative compactness result]
\thlabel{bound independant delta}
Let $\rho_0$ and $\rho_1$ in $\mathcal{M}_+(\Om)$ and consider $\mathcal{G}_{\rho_0}^{\rho_1} \defeq  \bigcup_{\delta >0} \arg \min \ifonc(\rho_0,\rho_1)$. For all $\mu=(\rho,\M,\Z)\in \mathcal{G}_{\rho_0}^{\rho_1}$, the following bounds, independent of $\delta$, hold
\begin{align}
 D_{BB}(\rho,\M) & \leq d^2_{gBB}(\rho_0,\rho_1)\, ,
\label{ineqonDBB}\\
 D_{FR}(\rho,\Z) & \leq d^2_{FR}(\rho_0,\rho_1).
\label{ineqonDFR}
\end{align}
Moreover, $\mathcal{G}_{\rho_0}^{\rho_1}$ is relatively compact.
\end{lemma}

\begin{proof}
Inequalities \eqref{ineqonDBB} and \eqref{ineqonDFR} are direct applications of \thref{weighted optim}-1, to the characterizations of $d_{gBB}$ and $d_{FR}$ given in \thref{origin limit metrics}.
As these inequalities do not depend on $\delta$, we can deduce from \thref{upperbounds} a uniform TV bound on the elements of $\mathcal{G}_{\rho_0}^{\rho_1}$. More explicitly, there exists $C>0$ such that for all $\mu\in \mathcal{G}_{\rho_0}^{\rho_1}$,
\[
|\mu|([0,1] \times \Om) < C
\]
This implies that $\mathcal{G}_{\rho_0}^{\rho_1}$ is relatively compact for the weak* topology.
\end{proof}

We can now state the main result of this section.

\begin{theorem}[Limit models]
\thlabel{limitmodels}
Let $\rho_0$ and $\rho_1$ in $\mathcal{M}_+(\Om)$ and let $(\rho^n,\M^n,\Z^n)_{n\in \N}$ be a sequence of minimizers for the distance $\WF_{\delta_n}$ between $\rho_0$ and $\rho_1$ with $(\delta_n)_n \in \R_+^{\N}$.

A. If  $\delta_n \underset{n \to \infty}{\to} +\infty$ then $(\rho^n,\M^n,\Z^n)_{n\in \N}$ weak* converges (up to a subsequence) to a minimizer for $d_{gBB}$.

B. If  $\delta_n \underset{n \to \infty}{\to} 0$ then $(\rho^n,\M^n,\Z^n)_{n\in \N}$ weak* converges (up to a subsequence) to a minimizer for $d_{FR}$.
\end{theorem}

\begin{proof}
Let us review the hypothesis of \thref{weighted optim} .
First, $D_{BB}$ and $D_{FR}$ are l.s.c. The existence of minimizers for $\ifonc$ has been shown in \thref{existence} for any positive value of $\delta$. Moreover, as a consequence of \thref{bound independant delta}, those minimizers are in a compact set---just take any compact set which contains the closure of the relatively compact set of all minimizers. All conditions are gathered for \thref{weighted optim}-(2,3) to be applied, hence the result.
\end{proof}





\section{Explicit Geodesics}
\label{sec-explicit-geod}

In this section, we give an explicit description of the behavior of geodesics in three cases: for an ``inflating'' measure, for the transport of one Dirac to another, and finally for the transport of multiple couples of Diracs. In order to establish explicit geodesics, we will first exhibit an ansatz and then prove the existence of an optimality certificate as defined in \thref{certificate}.

\subsection{No Transport Case: Inflating and Deflating Measures}

The case when there is no transport is dealt with in the following propositions. 
\begin{proposition}[A uniqueness result]
\thlabel{no transport case}
Let $(\rho_0,\rho_1) \in \mathcal{M}_+(\Om)^2$. If $(\rho,0,g\rho)$ is a minimizer of \eqref{dual}, where $g$ is the homogeneous rate of growth defined in \eqref{rate of growth},
then $\rho$ is the \emph{unique} geodesic.
\end{proposition}
\begin{proof}
Assuming $(\rho,0,g\rho)$ minimizes \eqref{dual}, we have $\WF_{\delta}^2(\rho_0,\rho_1)= \delta^2 D_{FR}(\rho,\Z)$. Moreover, $(\rho,0,g\rho) \in \argmin_{\mu \in \ccons} D_{FR}(\rho,\Z)$.
 Thus for any $\tilde{\mu}=(\tilde{\rho}, \tilde{\M}, \tilde{\Z}) \in \ccons$, $\ifonc(\tilde{\mu}) = D_{BB}(\tilde{\rho},\tilde{\M}) + \delta^2 D_{FR}(\tilde{\rho},\tilde{\Z}) \geq D_{BB}(\rho,\M) + \WF_{\delta}^2(\rho_0,\rho_1)$. This proves that for all minimizers, $\M=0$ and thus there is a unique geodesic which is the Fisher-Rao geodesic (see \thref{uniqueness FR}).
\end{proof}

\begin{proposition}[No transport case]
\thlabel{rescaled measures}
If $\rho_1=\alpha \rho_0$ with $\alpha \geq 0$, then 
\[ 
\rho=\left( t\sqrt{\alpha} + (1-t)\right)^2 \rho_0 \otimes \d t
\]
is the \emph{unique} geodesic for $\WF_{\delta}$, for all $\delta>0$, and
\[
\WF_{\delta}(\rho_0,\alpha \rho_0)= \delta \vert \sqrt{\alpha} - 1 \vert \sqrt{2\rho_0(\Om)}.
\]
\end{proposition}

\begin{proof}
 If $\alpha>0$, consider the function defined on $[0,1]\times \bar{\Omega}$ by
\[
\varphi(t,x)=\frac{2\delta^2(\sqrt{\alpha}-1)}{(t\sqrt{\alpha}+(1-t))} \, ,
\]
and write the set $\D^c \ifonc (\mu)$ :
\begin{multline}
\D^c \ifonc (\mu) = \Big\{ (\alpha,\beta,\gamma) \in C([0,1]\times \Om; B_{\delta}) :
 \alpha+\frac12 \frac{\gamma^2}{\delta^2} = 0 - \text{$\rho$ a.e.\, $\beta = 0 $ $-\rho$ a.e.\,  }\\
\text{and}\, \gamma=\frac{2\delta^2(\sqrt{\alpha}-1)}{(t\sqrt{\alpha}+(1-t))} - \text{$\rho$ a.e. }
 \Big\}.
\end{multline}
We can check that $(\D_t \varphi, \nabla \varphi, \varphi) \in \D^c \ifonc (\mu)$ because, in particular, $\nabla \varphi = 0$ and $\D_t \varphi+\varphi^2/(2\delta^2)=0$. Then, as $(\rho,\M,\Z)\in \ccons$ the ansatz is a geodesic. The distance is found by computing $\ifonc(\mu)$. \\

Now if $\alpha=0$ the certificate $\varphi$ we built above is not defined at $t=1$ and this approach is fruitless. However, by the triangular inequality, for any $\alpha>0$, 
\[
\WF_{\delta} (\rho_0,\alpha \rho_0) \leq \WF_{\delta} (\rho_0,0) + \WF_{\delta} (0,\alpha \rho_0).
\]
By \thref{bounddistance} we have, $\WF_{\delta}(0,\alpha \rho_0)\leq \delta \sqrt{2\alpha \rho_0(\Om)}$. We rearrange the inequalities to obtain
\[
\delta \vert \sqrt{\alpha} - 1 \vert \sqrt{2\rho_0(\Om)} - \delta \sqrt{2\alpha \rho_0(\Om)} \leq \WF_{\delta} (\rho_0,0) \leq \delta \sqrt{2\rho_0(\Om)}.
\]
As this inequality holds for all $\alpha > 0$, we obtain $\WF_{\delta} (\rho_0,0)=\delta\sqrt{2\rho_0(\Om)}$ which is the value of the functional evaluated at the ansatz $\mu$ for $\alpha=0$, so the proof that $\rho$ is a geodesic is complete. For uniqueness, it only remains to remark that $\D_t \rho = g \rho$ (with $g$ defined in \eqref{rate of growth}) and then to apply \thref{no transport case}.
\end{proof}

\subsection{Transport of One Dirac to Another}
\label{sec-travelling-dirac}

We first solve the case of the geodesic between two Diracs and then generalize this result to configurations with more Diracs. In the case of two Diracs, we show that if the Diracs are closer than $\pi \delta$, then the \emph{unique} geodesic is a travelling Dirac. Beyond that distance, the Fisher-Rao geodesic is the \emph{unique} geodesic.
\begin{theorem}
\thlabel{2diracs}
Consider two Diracs of mass $h_0$ and $h_1$ and location $x_0$ and $x_1$ respectively, i.e.\ $\rho_0=h_0\delta_{x_0}$ and $\rho_1=h_1\delta_{x_1}$. We distinguish 3 types of behaviors for geodesics of $\WF_{\delta}$:
\begin{enumerate}
\item \emph{Travelling Dirac.} If $\vert x_1 - x_0 \vert <\pi\delta$ then the travelling Dirac $\rho= h(t) \delta_{x(t)} \otimes \d t$ implicitly defined by
\[
\begin{cases}
h(t) = At^2-2Bt+h_0\\
h(t)x'(t)=\M_0
\end{cases}
\]
is the \emph{unique} geodesic. Here $\M_0$, $A$ and $B$ are constants explicitly depending on $h_0$, $h_1$ and $x_1-x_0$ (the relations are given in the proof). 
\item \emph{Cut Locus.} If $\vert x_1 - x_0 \vert =\pi \delta$ then there are infinitely many geodesics. Some of them can be described in the following way: take $N\in \N$ and choose two $N$-tuples $(h_0^i)_{i=1,\dots, N}$ and $(h_1^i)_{i=1,\dots, N}$ of non-negative real numbers satisfying $h_0=\sum_1^N h_0^i$ and $h_1=\sum_{i=1}^N h_1^i$. A geodesic is given by  $\rho= \sum_{i=1}^N \rho_i$ where $\rho_i$ is either the travelling Dirac of point 1.\ or a Fisher-Rao geodesic between $\rho_0^i \delta_{x_0}$ and $\rho_1^i \delta_{x_1}$. Notice that a single travelling Dirac or the Fisher-Rao geodesics are particular cases.
\item \emph{No transport.} If $\vert x_1 - x_0 \vert> \pi \delta$ then the Fisher-Rao geodesic 
\[
\rho = \left[ t^2h_1\delta_{x_1}+(1-t)^2h_0\delta_{x_0}\right] \otimes \d t
\]
is the \emph{unique} geodesic.
\end{enumerate}
\end{theorem}

\begin{corollary}
 If $|x_1-x_0|<\pi \delta$ and if $\Om \subset \R$, then the distance is (by formula \eqref{link distance A})
\begin{align*}
\WF_{\delta}(h_0\delta_{x(0)},h_1 \delta_{x(1)}) 
&= \sqrt{2}\delta \left[ h_0 + h_1 -2\sqrt{h_0 h_1} \cos \left( \frac{x_1-x_0}{2\delta} \right) \right]^{1/2}\\
&= \sqrt{2}\delta | \sqrt{h_1} e^{i x_1/(2\delta)} - \sqrt{h_0} e^{i x_0/(2\delta)} | \, .
\end{align*}
This induces a local isometric injection from the space of Dirac measures on $]0,\delta\pi[$ to $\mathbb{C}$ equipped with the flat Euclidean metric.
\end{corollary}

\begin{remark}
If $h_0=h_1=h$ and $|x_1-x_0|<\pi \delta$ then the minimum mass of the geodesic is attained at $t=1/2$ and its value is
\[
\frac h2 \left[ 1+\cos \left( \frac{|x_1-x_0|}{2\delta} \right) \right]
\] 
which is never smaller than $h/2$.
\end{remark}

\begin{proof}
The proof is divided into 4 parts, in the first three parts, $\Om$ is a segment in $\R$. First we derive an ansatz for the geodesics when $\vert x_1-x_0\vert<\pi\delta$, and write the sufficient optimality conditions. Then comes a part with technical computations where we prove the existence of an optimality certificate. In a third part we study what happens at the cut locus and beyond. Finally, we extend the proof for $\Om \in \R^d$, for any dimension $d$. 
\item 
\begin{paragraph}{i. Ansatz and optimality conditions}
In this part, we consider the case where $\vert x_1 - x_0 \vert < \pi\delta$ and $\Om$ is an open segment of $\R$. Let us look for geodesics of the form of travelling Diracs, i.e 
\begin{equation}
\label{travelling dirac}
\rho = h(t)\delta_{x(t)} \otimes \d t
\end{equation}
where $h$ and $x$ are functions from $\R$ to $\R$, assumed sufficiently smooth so that all oncoming expressions make sense. Moreover, we assume $h_0,h_1 \neq 0$ and $x_1\neq x_0$ (those special cases are all dealt with in \thref{rescaled measures}). Satisfying the continuity constraint imposes $\M = h(t)x'(t)\delta_{x(t)} \otimes \d t$ and $\Z = h'(t)\delta_{x(t)} \otimes \d t$, so the functional reads
\[
\ifonc(\rho,\M,\Z)=\int_0^1 \frac12 \left( x'(t)^2 h(t) +\delta^2 \frac{h'(t)^2}{h(t)} \right) \d t.
\]
The Euler-Lagrange conditions imply that minimizers among travelling Diracs should satisfy, for some $\M_0\in \R$,
\begin{eqnarray*}
\begin{cases}
h(t)x'(t)=\M_0 \\
2h''(t)h(t)-h'(t)^2=\M_0^2/\delta^2 \\
h(0) , h(1) , x(0) , x(1) \text{ fixed.}
\end{cases}
\end{eqnarray*}
Solutions of the second order differential equation are of the form $h(t)=At^2-2Bt+h_0 $ where $A$ and $B$ are given by the system 
 \[
 \begin{cases}
A h_0-B^2 = \frac{\M_0^2}{4\delta^2} \\
A - 2B = h_1 - h_0
\end{cases}
 \]
 This is a second order equation which admits two solutions
\[
A =  h_1 +h_0 - 2 \epsilon \sqrt{h_0h_1-\frac{1}{4} \frac{\M_0^2}{\delta^2}} 
\quad \text{and} \quad 
B =  h_0 - \epsilon \sqrt{h_0h_1-\frac{1}{4} \frac{\M_0^2}{\delta^2}}
\]
where $\epsilon \in \{ -1,+1\}$. In order to choose the value for $\epsilon$, we plug the expressions in the functional
\begin{equation}
\label{link distance A}
\ifonc(\rho,\M,\Z)=\int_0^1 \frac12 \frac{\M_0^2+\delta^2 h'(t)^2}{h(t)} \d t = 2 \delta^2 A
\end{equation}
and conclude that $\epsilon=+1$ as otherwise, the functional is greater than the Fisher-Rao upper bound given in \thref{bounddistance}. Also, \eqref{link distance A} shows that $A>0$ since $\WF_{\delta}(\rho_0,\rho_1)>0$. In order to determine $\M_0$, we integrate the speed to obtain
\[
\frac{x_1-x_0}{\M_0}
=\int_0^1 \frac{\d t}{h(t)}
=\frac{2\delta}{\M_0} \int_{-\alpha}^{\beta} \frac{\d u}{1+u^2}
=\frac{2\delta}{\M_0} \left[ \arctan(\beta)+ \arctan(\alpha) \right]
\]
with 
\begin{equation}
\label{alpha,beta}
\alpha \defeq \frac{2\delta}{\M_0}B 
\quad \text{and} \quad 
\beta \defeq \frac{2\delta}{\M_0}(A-B).
\end{equation}
The expression for the sum of $\arctan$ depends on the sign of $1-\alpha \beta$. Yet we find
\[
1-\alpha \beta = \frac{4\delta^2}{\M_0^2}A\sqrt{h_0h_1-\frac{\M_0^2}{4\delta^2}}
\]
which is positive since $A>0$. So we are in the case 
\[
\arctan \alpha + \arctan \beta = \arctan \left( \frac{\alpha+\beta}{1-\alpha \beta} \right) \in [-\pi/2,\pi/2]
\]
and solutions exist if and only if $|x_1-x_0|< \pi \delta$. We introduce $\tau \defeq \frac{\alpha + \beta}{1- \alpha \beta}$ and by direct computations
\[
\M_0 = 2 \delta \tau \sqrt{\frac{h_0 h_1}{1+\tau^2}}
\quad \text{ and } \quad 
\tau = \tan \left( \frac{x_1-x_0}{2\delta} \right).
\]
Notice that we can rewrite $A$ and $B$ in a simpler way with the new parameter $\tau$
\begin{eqnarray*}
A= h_1 +h_0 - 2\sqrt{\frac{h_0 h_1}{1+\tau^2}} \geq 0 & \text{and} & B= h_0 -  \sqrt{\frac{h_0 h_1}{1+\tau^2}}
\end{eqnarray*}
making it clear that $A$ and $B$ are well defined.

\end{paragraph}
\item
\begin{paragraph}{ii. Existence of a certificate}
So far, we have shown that travelling Diracs of the form \eqref{travelling dirac} cannot be geodesics if $|x_1-x_0|>\pi \delta$. Otherwise, there is a unique candidate solution and this part is devoted to the construction of an optimality certificate in order to show that it is optimal over all measure valued trajectories and not only among travelling Diracs. For simplicity, let us take $\delta=1$ from now on, without loss of generality thanks to \thref{rescaling}. By \thref{certificate}, $\varphi \in C^1([0,1]\times \Om)$ is an optimality (and uniqueness) certificate if for all $(t,x)\in [0,1]\times \bar{\Om}$
\begin{equation}
\D_t \varphi +\frac12 \left( (\D_x \varphi) ^2 + \varphi^2 \right) \leq 0
\label{ineqcondition}
\end{equation}
and on the support of $\rho$ :
\begin{numcases}{}
\D_x \varphi(x(t),t) = x'(t) \label{dualconda}\\
\varphi(x(t),t)= \frac{h'(t)}{h(t)} \label{dualcondb}\\
\D_t \varphi +\frac12 \left( (\D_x \varphi )^2 + \varphi^2 \right) = 0 \label{dualcondc} .
\end{numcases}
In view of the equations $\varphi$ should satisfy, we suggest the following optimality certificate for the candidate $(\rho,\M,\Z)$:
\[
\varphi(t,x)=a(t) \cos(x-\theta) + b(t)
\]
Inequality \eqref{ineqcondition} gives, for all $(t,x)\in [0,1]\times \bar{\Om}$:
\[
\left( a'(t) +a(t)b(t) \right) \cos(x-\theta) + \left(  b'(t) +\frac12 ( a^2(t) +b^2(t) ) \right) \leq 0
\]
As equality has to be reached at all points of the form $(x(t),t)$ from condition \eqref{dualcondc}, the only way to satisfy this inequality is to have
\[
\begin{cases}
a' + ab = 0\\
b' + \frac12(a^2+b^2)=0.
\end{cases}
\]
The solutions to this differential system are of the form
\[
\begin{cases}
a(t)=\frac{1}{t-t_1} -\frac{1}{t-t_2}\\
b(t)=\frac{1}{t-t_1} +\frac{1}{t-t_2}
\end{cases}
\]
Now assume $t_1, t_2 \notin [0,1]$, so that $a$ and $b$ are $C^{\infty}$ on $[0,1]$. This assumption will be checked below.

It remains to check conditions \eqref{dualconda} and \eqref{dualcondb}. Equation \eqref{dualconda} imposes $x'(t)=-a(t)\sin (x(t)-\theta)$. We first determine $x$ implicitly from this equation before checking in a later step that it indeed corresponds to $\M_0/h$. For $k\in \mathbb{Z}$, the function $x : t\in [0,1] \mapsto \theta +k\pi$ is a stationary solution. As $a$ is smooth, the Picard-Lindel\"of theorem guarantees the uniqueness of the solution of any initial value problem, so solutions do not intersect. Since $x$ is not constant, we have $\forall t \in [0,1], x(t)\notin \{ \theta + k\pi, k \in \mathbb{Z} \}$, and thus $\sin(x(t)-\theta)$ does not vanish in $[0,1]$. Let us divide by $\sin(x(t)-\theta)$ and integrate equation \eqref{dualconda}
\begin{eqnarray*}
\frac{x'(t)}{\sin(x(t)-\theta)} = \frac{1}{t-t_2} -\frac{1}{t-t_1}
&\Leftrightarrow&
 \frac12 \log \frac{1-\cos (x(t)-\theta)}{1+\cos (x(t)-\theta)} = \log \frac{1}{\kappa}  \left\vert \frac{t-t_2}{t-t_1} \right\vert
\end{eqnarray*}
and thus
\begin{equation}
\label{eq: implicit x prime}
\cos (x(t)-\theta) = \frac{\kappa^2(t-t_1)^2 -(t-t_2)^2}{\kappa^2(t-t_1)^2 +(t-t_2)^2}\\
\end{equation}
for some $\kappa > 0$. The value of the integration constant $\kappa$ will be studied later on, but we already see that $\theta$ is determined given $t_1$ and $t_2$:
\[
\theta= x(0) - \arccos \left( \frac{\kappa^2t_1^2-t_2^2}{\kappa^2t_1^2+t_2^2}\right)\; \text{mod } 2\pi
\]
The last constraint we have is equality \eqref{dualcondb}: $a(t)\cos (x(t)-\theta) +b(t)= \frac{h'(t)}{h(t)}$.
After several lines of re-arranging one finds this is equivalent to
$$
2\frac{\kappa^2(t-t_1)+(t-t_2)}{\kappa^2(t-t_1)^2+(t-t_2)^2}=2\frac{At-B}{At^2-2Bt+h_0} \, .
$$
By identification, there exists a factor $\lambda_2 >0$ such that:
\begin{equation*}
\begin{cases}
A= \lambda_2(\kappa^2+1)\\
B= \lambda_2(\kappa^2 t_1+t_2)\\
h_0=\lambda_2(\kappa^2 t_1^2+t_2^2).
\end{cases}
\end{equation*}
This can be symmetrized as
\begin{equation}
\begin{cases}
A= \lambda_1+\lambda_2\\
B= \lambda_1 t_1+\lambda_2 t_2\\
h_0=\lambda_1 t_1^2+\lambda_2 t_2^2.
\end{cases}
\label{diraccondition}
\end{equation}
with $\lambda_1>0 , \, \lambda_2 > 0$ (we imposed $\lambda_1=\lambda_2 \kappa^2$).

%

Also, note that gathering conditions \eqref{dualconda},  \eqref{dualcondb} and \eqref{dualcondc} leads to the relation $2h'h-(h')^2 = (x')^2h^2$ and thus, by the Euler-Lagrange conditions satisfied by $h$, $(x')^2h^2 = \omega_0^2$. This proves that the function $x$ implicitly defined by \eqref{eq: implicit x prime} is the same as that of the ansatz.
To sum up, if a solution to system \eqref{diraccondition} satisfies the constraints
\[
\begin{cases}
\lambda_1> 0\\
\lambda_2 > 0\\
(t_1,t_2) \notin [0,1]^2\\
\end{cases}
\]
then we have built an optimality certificate. System \eqref{diraccondition} leads to the following candidate solutions:
\[
\begin{cases}
t_1 = \frac{B}{A} +\varepsilon \frac{\M_0}{2A} \frac{1}{ \kappa} \\
t_2 = \frac{B}{A} -\varepsilon \frac{\M_0}{2A} \kappa\\
\lambda_2=A-\lambda_1
\end{cases}
\]
where $\kappa=\sqrt{\frac{\lambda_1}{\lambda_2}}$ and $\varepsilon \in \{-1,1\}$. The sign of $\varepsilon$ depends on direction of the movement. We choose say, $\varepsilon=1$, the other case is similar. Let us look for a solution satisfying, e.g. $t_1>1$ and $t_2<0$. This leads to the conditions
\begin{eqnarray*}
\kappa > \alpha 
& 
\text{and} & 
\frac{1}{\kappa} > \beta
\end{eqnarray*}
where $\alpha$ and $\beta$ have been defined earlier in the proof in \eqref{alpha,beta}. It is not excluded that $\beta$ or $\alpha$ be non-positive, in that case it is easy to find $\kappa$ satisfying both inequalities. Now, if $\alpha, \beta >0$, the necessary and sufficient condition for a solution $\kappa$ to exist, is $\alpha \beta <1$, which we have already checked. This concludes the proof for the case $\vert x_1-x_0 \vert < \delta \pi$.
\end{paragraph}
\item
\begin{paragraph}{iii. The cut locus and beyond}
Now we study what happens when $\vert x_1 -x_0 \vert= \pi \delta$, which is the cut locus distance. 
Introduce a Dirac $h_1 \delta_{x(s)}$ between $h_0\delta_{x_0}$ and $h_1\delta_{x_1}$ where $x(s)=x_0 + s (x_1-x_0)$. It is not hard to check (see for instance equation \eqref{link distance A} ) that 
\[
\WF^2_{\delta}(h_0\delta_{x_0},h_1\delta_{x(s)}) \underset{s \rightarrow 1}{\rightarrow} 2\delta^2 (h_0+h_1)
\]
 and that 
 \[
 \WF^2_{\delta}(h_1\delta_{x(s)},h_1\delta_{x_1}) \underset{s \rightarrow 1}{\rightarrow} 0 .
 \]
Combining two triangle inequalities---for the upper and lower bounds---we obtain
\[
\WF^2_{\delta}(h_0\delta_{x_0},h_1\delta_{x_1}) = 2\delta^2(h_0+h_1).
\]
This has important consequences. First, it proves that the travelling Dirac is still a geodesic at the cut locus, with the trajectory $\rho_t=h(t)\delta_{x(t)}$ where
\begin{eqnarray*}
h(t)=h_0(t-1)^2+h_1t^2& \text{and} &
x'(t) = \frac{4 h_0 h_1 \delta}{h(t)}
\end{eqnarray*}
since $\ifonc$ evaluated at this trajectory is worth $2\delta^2(h_0+h_1)$.
Second, as $\ifonc$ linearly depends on the masses we can build infinitely many geodesics, examples are described in the theorem. Finally, the cost at the cut locus equals the upper-bound of pure Fisher-Rao (recall \thref{bounddistance}).

Thus, by monotonicity of the distance w.r.t.\ rescaling (see \thref{monotonicity of the distance}), the $\WF_{\delta}$ distance between two Diracs for which $|x_1-x_0|\geq \pi \delta$ is $\delta\sqrt{2(h_0+h_1)}$ and a geodesic is
\[
\rho = \left[ h_1 t^2\delta_{x_1}+h_0(1-t)^2\delta_{x_0}\right] \otimes \d t \, .
\]
For uniqueness, we again apply \thref{monotonicity of the distance} for proving that, when $|x_1-x_0|> \pi \delta$, the component $\omega$ of minimizers is zero. Thus the only geodesic is the Fisher-Rao geodesic, by \thref{uniqueness FR}.
\end{paragraph}
\item

\begin{paragraph}{iv. Extension to arbitrary dimension}
The extension of the previous results for $\Om \in \R^d $, $d>1$ is rather straightforward and we will do it in such a way so as to prepare the ground for the next theorem. Recall that what we want to prove (and have proven for $d=1$ so far) is that if $\vert x_1 -x_0 \vert < \pi \delta$, then the travelling Dirac (as described in the theorem) is the unique geodesic. To show this, we choose $\delta=1$ without loss of generality and we introduce the function
\begin{equation} 
\varphi(t,x)=\left( \frac{1}{t-t_1}-\frac{1}{t-t_2} \right) \cos(\vert x-\theta \vert ) + \frac{1}{t-t_1}+\frac{1}{t-t_2} 
\label{certifddim}
\end{equation}
where the constants $t_1$, $t_2$ and the vector $\theta$ can be found by the same method than above by considering the problem on the line $\text{span}\{x_1 - x_0 \}$. This is possible because cosine is an even function and $\theta \in \text{span}\{x_1 - x_0 \}$. Now remark that the derivative of $\cos(\vert \cdot \vert)$ is continuous at $0$, thus $\varphi \in C^1([0,T]\times\bar{\Om})$ and by construction satisfies all the conditions to be an optimality and uniqueness certificate and the result is shown. The behaviors at the cut locus and beyond extend straightforwardly to this case.
\end{paragraph}
\end{proof}

\subsection{Matching Atomic Measures}

Geodesics between one couple of Diracs are so far quite well-understood. To go one step further in the comprehension of the behavior of our metric, we study the geodesics between two purely atomic measures $\rho_0$ and $\rho_1$. We show that under some geometrical conditions on the locations of the atoms, the behavior is similar to that of a pair of atoms. Those conditions require basically that the atoms be arranged by close pairs of Diracs, and the pairs should be far from each other. 

\begin{theorem}
\thlabel{pairsdiracs}
Let $\rho_0 = \sum_{i=1}^N h_0^i \delta_{x^i_0}$ and $\rho_1 = \sum_{i=1}^N h^i_1 \delta_{x^i_1}$ be the initial and final measures, where $h^i_0$ or $h^i_1$ may be zero. Also suppose that 
\begin{itemize}
\item $\forall i \in \{1,\dots N \}, \vert x^i_0-x^i_1 \vert < \pi \delta$,
\item $\forall (i,j) \in \{1,\dots N \}^2, i\neq j \Rightarrow \max_{\alpha,\beta \in \{0,1\}}\{\vert x^i_{\alpha}-x^j_{\beta} \vert\} > 6\pi \delta$. 
\end{itemize}
Then $\rho = \sum_{i=1}^N \rho^i$, where for all $i\in \{1, \dots, N \}$, $\rho^i = h^i(t) \delta_{x^i(t)} \otimes \d t$ is the geodesic between $h^i_0 \delta_{x^i_0}$ and $h^i_1 \delta_{x^i_1}$ from \thref{2diracs}, is a geodesic for $\WF_{\delta}$. This geodesic is \emph{unique} if all masses $h_0^i$, $h_1^i$ are positive.
\end{theorem}

\begin{remarks} \mbox{}
\begin{itemize}
\item A finer analysis could easily lower the ``security distance'' of $6\pi \delta$ between pairs of Diracs so that they can be considered independently.
\item In figure \ref{fig: matchingdiracs} we give an example where the hypotheses of \thref{pairsdiracs} are satisfied.
\item We only prove uniqueness of the geodesic when the certificate is non-degenerate. But it is likely that uniqueness hold in general under our hypotheses.
\end{itemize}
\end{remarks}

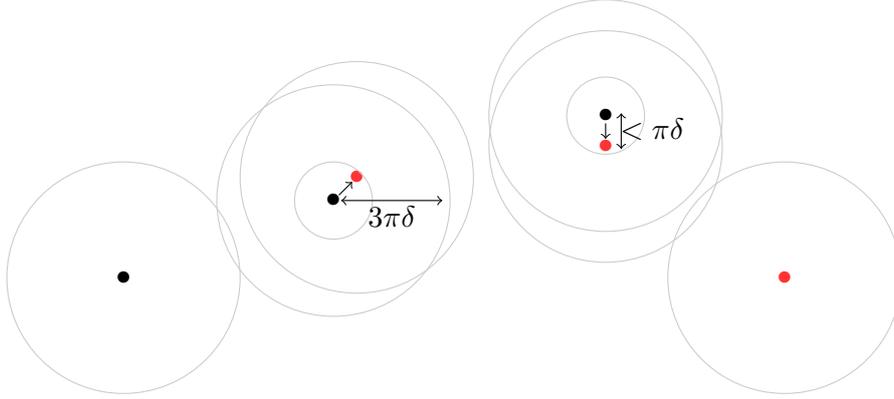
\begin{figure}
 \centering
  \resizebox{0.90\linewidth}{!}{
\begin{tikzpicture}
       	\draw[black] (0,0) node {$\bullet$} ;
	\draw[red!80] (.3,.3) node {$\bullet$} ; 
	\draw[->] 	(.07,.07)    	-- 	(.24,.24);
	\draw[black!20] (0,0) circle (1.5);
	\draw[black!20] (0,0) circle (0.5);
	\draw[black!20] (.3,.3) circle (1.5);
	\draw[<->]	(0.1,0)    	-- 	(1.4,0);
	\draw (0.75,-0.2)  node {$3\pi \delta$};
	\draw[black] (3.5,1.1) node {$\bullet$} ;
	\draw[red!80] (3.5,0.7) node {$\bullet$} ; 
	\draw[->] 	(3.5,1.0)    	-- 	(3.5,0.8);
	\draw[black!20] (3.5,1.1) circle (1.5);
	\draw[black!20] (3.5,1.1) circle (0.5);
	\draw[black!20] (3.5,.7) circle (1.5);
	\draw[<->] 	(3.7,1.13)    	-- 	(3.7,0.67);
	\draw (4.1,0.92)  node {$<\pi \delta$};
	\draw[black] (-2.7,-1) node {$\bullet$} ;
	\draw[black!20] (-2.7,-1) circle (1.5);
	\draw[red!80] (5.8,-1) node {$\bullet$} ;
	\draw[black!20] (5.8,-1) circle (1.5);
\end{tikzpicture}
        }
\caption{Configuration in dimension $2$ satisfying hypotheses of \thref{pairsdiracs}. The initial $\rho_0$ and final $\rho_1$ measures are composed of the black and red Diracs, respectively. At least 1 couple of spheres of radius $3\pi \delta$ do not intersect for each pair of systems. We show that in that case, the geodesic is the sum of two travelling Diracs (for the pairs), an on-place decreasing Dirac (single black) and an on place increasing Dirac (single red).}
\label{fig: matchingdiracs}
\end{figure}

\begin{proof}
Again we choose $\delta=1$ without loss of generality. As for now, assume that all masses $h^i_0, h^i_1$ are strictly positive, the case of vanishing masses, i.e.\ ``single'' Diracs, will be fixed at the end of the proof. We aim at building an optimality certificate $\varphi \in C^1([0,1]\times \bar{\Om})$ for the candidate geodesic of the theorem. Hence $\varphi $ should satisfy, for all $(t,x)\in [0,1]\times \bar{\Om}$,
\[
\D_t \varphi +\frac12 \left( |\nabla \varphi| ^2 +  \varphi^2 \right) \leq 0
\]
and for all $t\in [0,1]$, for all $i\in \{ 1, \dots , N \}$ :
\[
\quad 
\begin{cases}
\nabla \varphi(x^i(t),t) = (x^i)'(t) \\
\varphi(x^i(t),t)=(h^i)'(t)/h^i(t) \\
\D_t \varphi +\frac12 \left( |\nabla \varphi |^2 + \varphi^2 \right) = 0 \quad  \text{for points of the form $(t,x^i(t))$}.
\end{cases}
\]

We introduce for $i\in \{1,\dots N\}$, the certificates $\varphi^i$ for the geodesics between each couple of Diracs $(h^i_0 \delta_{x^i_0}, h^i_1 \delta_{x^i_1})$ as defined in equation \eqref{certifddim}. The three parameters describing $\varphi^i$ are denoted $t^i_1, t^i_2$ and $\theta^i$. As $\theta^i$ is only defined up to translations, we decide from now on that $\theta^i$ is the unique such translation satisfying $\max\{ \vert x^i_0-\theta^i \vert, \vert x^i_1-\theta^i \vert \}< \pi$, which exists under our assumption that $|x_0^i - x_1^i|<\pi$. Thus, from the hypotheses in the theorem, we have that $\min_{i \neq j}\{\vert \theta^i- \theta^j \vert \} > 4 \pi$. Now consider the \emph{binding} functions, defined for $i \in \{1,\dots,N \}$ by
\[
\varphi^i_b(t,x)=-\frac{1}{t-t^i_2} \cos(\vert x-\theta^i \vert ) + \frac{1}{t-t^i_2}. 
\]
Each $\varphi^i_b$ is in phase with $\varphi^i$ and satisfies 
\[
\varphi^i_b(t,\pi u + \theta^i)=\frac{2}{t-t_2^i}=\varphi^i(t,\pi u+ \theta^i) 
\quad \text{ and } \quad
\varphi^i_b(t,2\pi u+ \theta^i)=0
\]
for all $u\in \R^d$, $\vert u \vert = 1$ and $t\in[0,1]$. Also they are solutions to the equation $\D_t \varphi^i_b +\frac12 \left( (\nabla \varphi^i_b )^2 + (\varphi^i_b)^2 \right) = 0$. Now, as the constant zero function is a solution of this p.d.e too, we introduce
\begin{equation*}
\varphi(t,x)=
\begin{cases}
\varphi^i (t,x) & \text{if }  \vert x - \theta^i \vert \leq \pi \\
\varphi^i_b (t,x) & \text{if }  \vert x - \theta^i \vert \leq 2\pi  \text{ and } \vert x - \theta^i \vert > \pi\\
0 & \text{ otherwise.}
\end{cases}
\end{equation*}
Under the hypotheses of the theorem, only one condition is met at a time so this function is well defined. It belongs to $C^1([0,1]\times \bar{\Om})$ because the bindings are located where the gradients vanish, on the extrema of cosines. See Figure \ref{figure certificate} to picture how the bindings are performed between the $\varphi^i$s. By construction, all hypotheses are fulfilled so that $\varphi$ be an optimality and uniqueness certificate.

Now if one or more of the $h^i_0$ or $h^i_1$ vanish, the same reasoning as in the proof of \thref{2diracs} allows to compute the distance between $\rho_0$ and $\rho_1$ (by an upper and lower bound, making the mass of vanishing Diracs tend to zero), and all that remains is to see that replacing the travelling Dirac by a Fisher-Rao geodesic for single Diracs reaches that cost. 
\end{proof}

\begin{figure}[ht]
\centering
 \includegraphics[trim=3cm 3cm 3cm 3cm, width=.7\linewidth]{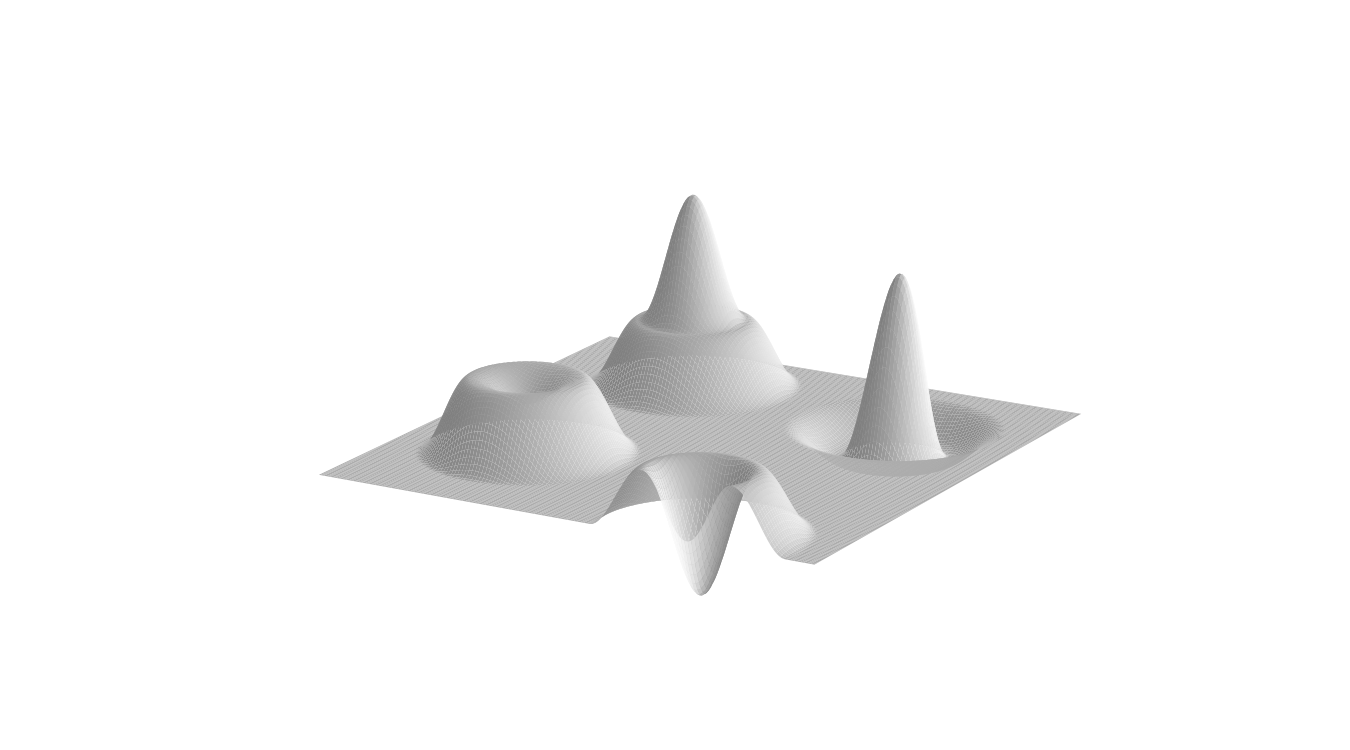}
 \caption{Example of an optimality certificate for $4$ pairs of Diracs, in the case $d=2$ at fixed $t$. The centers of the bumps are the $\theta^i$s. At the position $(t,x_i(t))$ of a travelling Dirac, the gradient gives its speed and the height gives its rate of growth.} \label{figure certificate}
\end{figure}



\section{Numerical Results}
\label{sec-numerics}

This section presents a numerical implementation of the computation of geodesics for our metric. The source code to reproduce these results is available online\footnote{\url{https://github.com/lchizat/optimal-transport/}}.

\subsection{Discretization}

The numerical resolution of the dynamical optimal transport problem is often performed with first order methods using proximal splitting algorithms (\cite{benamou2000computational}, \cite{papadakis2014optimal}).  These methods extend with no difficulties to the case of transport with sources, as long as the proximal operator of the new functional can be easily computed. In order to compare a few models of optimal transport with source, we implemented a Douglas-Rachford proximal splitting algorithm on a staggered grid, in the footsteps of \cite{papadakis2014optimal}. The introduction of the source term comes with a few adjustments so we describe the numerical method that we implemented, in the case of a transport problem in 1-D for simplicity (but the scheme works in arbitrary dimension).

\paragraph{Centered and staggered grids.}

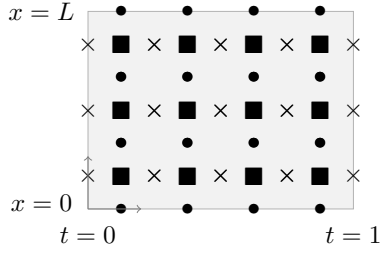
\begin{figure}
 \centering
  \resizebox{0.40\linewidth}{!}{
\begin{tikzpicture}
	\filldraw[color=black!30, fill=black!5] (-.5,-.5) rectangle (3.5,2.5);
	\foreach \x in {0,...,3}
   	 \foreach \y in {0,...,2} 
	 {
	\pgfmathparse{\x+0.5}\let\a\pgfmathresult
	\pgfmathparse{\y+0.5}\let\b\pgfmathresult
	\pgfmathparse{\x-0.5}\let\c\pgfmathresult
	\pgfmathparse{\y-0.5}\let\d\pgfmathresult
       	\draw[black,minimum size=5pt] (\x,\y) node {$\blacksquare$} ;
	\draw (\a,\y) node {$\times$} ; 
	\draw  (\x,\b) node {$\bullet$} ; 
	\draw  (\c,\y) node {$\times$} ; 
	\draw  (\x,\d) node {$\bullet$} ; 
	}
	
	\draw[->, black!50] 	(-.5,-.5)    	-- 	(-.5,.3);
	\draw	(-1.2,-.4) node {$x=0$};
	\draw	(-1.2,2.5) node {$x=L$};
	
	\draw[->,black!50] 	(-.5,-.5) 	-- 	(.3,-.5);
	\draw 	(-.5,-.9)  	node {$t=0$};
	\draw 	(3.5,-.9) 	node {$t=1$};
\end{tikzpicture}
        }
\caption{($\blacksquare$) centered and ($\times$, $\bullet$) staggered grids. The grey rectangle is the space-time domain and here $N=3$, $T=4$.}
\label{fig: grids}
\end{figure}


We consider the space domain $[0,L]$ and the time domain $[0,1]$. We denote $N$ the number of dicretization points in space and $T$ the number of discretization points in time. The centered grid is a even discretization of the interior of the space-time domain $[0,L] \times [0,1]$, namely
\[
\mathcal{G}_c = \left\{ \left( x_i=L\frac{i-1/2}{N}, t_j=\frac{j-1/2}{T} \right): 1 \leq i \leq N, \, 1 \leq j \leq T \right\}
\]
and the variable discretized on the centered grid are denoted
\[
V = (\rho,\M,\Z) \in \mathcal{E}_c = (\R^{\mathcal{G}_c })^3.
\]
Remark that the boundaries are not included in the centered grid. There are as many staggered grids as there are dimensions, defined as
\[
\mathcal{G}_s^x = \left\{ (x_i = L\frac{i-1}{N}, t_j=\frac{j-1/2}{T}) :  1\leq i\leq N+1, \, 1\leq j \leq T \right\}\, ,
\]
\[
\mathcal{G}_s^t = \left\{ (x_i = L\frac{i-1/2}{N}, t_j=\frac{j-1}{T}) :  1\leq i\leq N, \, 1\leq j \leq T+1 \right\}\, , 
\]
and the variables discretized on the staggered grid are denoted by
\[
U = (\bar{\rho},\bar{\M},\bar{\Z}) \in \mathcal{E}_s = \R^{\mathcal{G}_s^t }\times \R^{\mathcal{G}_s^x }\times \R^{\mathcal{G}_c } .
\]
The source $\Z$ is still discretized on the centered grid because it intervenes as is in the constraint (it is not differentiated).

Remark now that the continuity constraint is defined for a variable carried by the staggered grid while the functional is defined for a variable carried by the centered grid. One way out is to keep two variables $(U,V)\in \mathcal{E}_s \times \mathcal{E}_c$ (one on each grid) and to add an interpolation constraint between them, i.e.\ we require $V=\mathcal{I}(U)$ where $I$ is the midpoint interpolation operator 
\[
I(U)_{i,j} = \left( \frac{\bar{r}_{i,j+1}+\bar{r}_{i,j}}{2},\frac{\bar{m}_{i+1,j}+\bar{m}_{i,j}}{2},\bar{z}_{i,j} \right) \, .
\]

The discrete convex optimization problem that we aim to solve as an approximation of the continuous one reads then
\begin{equation*}
\min_{U,V} \,  \ifonc(V) + \iota_{\mathcal{CE}}(U) + \iota_{\{V=\mathcal{I}(U)\}} (U,V)
\end{equation*}
where the three terms represent the discrete functional, the indicator of the continuity constraint and the indicator of the interpolation constraint respectively. The discretization of the first two terms is carefully described in the next paragraphs.

\paragraph{Discrete continuity constraint.}
The discrete continuity constraint set with boundary conditions is defined as
\begin{equation}
\mathcal{CE} = \left\{   U= (\bar{\rho},\bar{\M},\bar{\Z}) \in \mathcal{E}_s :  AU = f_0 ,\, A = \left[
\begin{array}{c} 
\div - s_z \\ \hline 
s_b
\end{array}
\right]  \text{  and } f_0 = \left[
\begin{array}{c}  
0  \\ \hline 
b_0
\end{array}
\right]  \right\}
\end{equation}
where $s_z : U \mapsto \bar{\zeta}$, $s_b$ selects the values of $U$ on the boundaries of the staggered grids and $b_0$  is a vector containing $\rho_0$, $\rho_1$ and zeros for the Neumann boundary conditions. The divergence operator $\div$ only acts on the components $(\bar{\rho},\bar{\M})$ of $U$, takes its values on the centered grid and is defined for $U\in \mathcal{E}_s$ as 
\[
\div(U)_{i,j} = \frac{N}{L} (\bar{\M}_{i+1,j}-\bar{\M}_{i,j}) + T (\bar{\rho}_{i,j+1}-\bar{\rho}_{i,j}) \, .
\]

\paragraph{Discrete functional.}

First, we use \thref{rescaling} and bring ourselves back to the case $\delta = 1$ by rescaling the space ,  as this allow to obtain a simpler proximal operator. The discrete action is defined as $\ifoncN(V) \defeq \sum_{k\in\mathcal{G}_c } f(\rho_k,\M_k,\Z_k)$ with 
\[
f(\rho_k,\M_k,\Z_k) \defeq \frac{|\M_k|^2 + \Z_k}{2\rho_k} \, .
\]

\subsection{Minimization Algorithm}

\paragraph{Proximal methods and the Douglas Rachford algorithm.}

A popular class of first order methods for solving non-smooth convex optimization problems, so-called proximal splitting schemes, replaces the explicit gradient descent step by an implicit descent step using the so-called proximal operator. For a proper, convex, l.s.c.\ function $F$ defined on a Hilbert space $\mathcal{H}$ with values in $\R\cup +\infty$, the proximal operator is the single-valued map defined as
\[
	\prox_{\gamma F}(x) = \argmin_{\bar{x}\in \mathcal{H}} \frac12 \Vert x - \bar{x}  \Vert^2 + \gamma F(\bar{x}) \, .
\]
If $F$ is smooth, the optimality condition imply that $\bar{x} \defeq \prox_{\gamma F}(x)$ should satisfy $\bar{x} = x-\gamma \nabla F(\bar{x})$, justifying the common interpretation of the proximal operator as an implicit gradient step. For the indicator function of a convex set, the proximal operator is the projection on this set. For more complicated functions, it might be as difficult to compute the proximal operator as solving the whole minimization problem. Fortunately, if the function is the sum of simpler terms, minimization can be performed by so-called \emph{proximal splitting} methods. There are several ways to combine proximal operators in order to minimize the sum of functions. The one we implemented is the Douglas Rachford algorithm which allows to minimize the sum of two convex, proper, l.s.c.\ functions $G_1$ and $G_2$ with the following iterative scheme: choose two initial points $(z^{(0)},w^{(0)})\in \mathcal{H}^2$ and define the sequence
\begin{align*}
	w^{(l+1)} &= w^{(l)}+ \alpha (\prox_{\gamma G_1}(2z^{(l)}-w^{(l)})-z^{(l)}), \\
	z^{(l+1)} &= \prox_{\gamma G_2} (w^{(l+1)}) \, .
\end{align*}
If $0<\alpha<2$ and $\gamma >0$, one can show thant $z^{(l)} \rightarrow z^*$ a minimizer, see \cite{combettes2007douglas}. Depending on which functions we take as $G_1$ and $G_2$, we obtain various algorithms; in our code we used
\[
	G_1(U,V) = \iota_{\mathcal{CE}}(U) + \ifoncN(V) \quad \text{and} \quad G_2(U,V) = \iota_{V=\mathcal{I}(U)}(U,V) \, . 
\]
For a detailed review of proximal splitting algorithms, we refer the reader to \cite{combettes2011proximal}. Let us now compute the proximal operators.

\paragraph{Computing $\prox_{\gamma \ifoncN}$.}

The proximal operator of the functional $\ifoncN$ can be computed in closed form. The following result is a direct adaptation of \cite{papadakis2014optimal}.

\begin{proposition}
One has
\[
\forall V \in \mathcal{E}_c, \quad \prox_{\gamma \ifoncN}(V) = \left( \prox_{\gamma \foncN}(V_k)\right)_{k \in \mathcal{G}_c}
\]
where, for all $(\tilde{\rho},\tilde{\M}, \tilde{\Z}) \in \R \times \R^d \times \R$,
\[
\prox_{\gamma \foncN}(\tilde{\rho},\tilde{\M}, \tilde{\Z})  =
\begin{cases}
 \left( \rho^*, \frac{\rho^* \, \tilde{\M}}{\rho^* + \gamma},\frac{\rho^* \, \tilde{\Z}}{\rho^* + \gamma} \right) & \text{if $\rho^*>0$,}\\
 ( 0, 0, 0) & \text{otherwise},
\end{cases}
\]
and $\rho^*$ is the largest real root of the third order polynomial equation in $X$
\[
 P[X] = (X-\tilde{\rho})(X+\gamma)^2 - \frac{\gamma}{2} (|\tilde{\M}|^2+\tilde{\Z}^2) \, .
 \]
\end{proposition}

\paragraph{Computing $\proj_{\mathcal{CE}}$.}

The proximal operator of $\iota_{\mathcal{CE}}$ is the projection on the affine set $\mathcal{CE}$, which can be computed for $U \in \mathcal{E}_s$ as
 \begin{eqnarray}
\mathcal{P}_\mathcal{CE} (U) & = & U + A^* (A A^*)^{-1}(b_0 -AU) \\
& = & \mathcal{P}_B (U) - (Id-I_b)(s_z^* -\div^*) S^{-1} (s_z - \div ) (\mathcal{P}_B (U))
\end{eqnarray}
where $S=(\div - s_z)(Id - s_b^* s_b)(-\nabla - s_z^*)$ is the Schur complement of $AA^*$, $I_b$ is the identity on the boundaries on zero everywhere else and $\mathcal{P}_B$ is the projection on the boundary constraints. Given $p$, we can find $u$ satisfying $Su=p$ by solving $\Delta u -u + p=0$ on the centered grid with Neumann boundary conditions (on the staggered grids). In the Fourier domain, this equation reads
\[
\left[ 1 + (2- 2\cos( \pi m /N) )N/L + (2- 2\cos( \pi n/T) )T \right]   \hat{u}[m,n] = \hat{p}[m,n] \, .
\]
The DCT-II transform (with inverse DCT-III ) which coefficients are given by
\[
\hat{u}[k] = \sum_{n=1}^{N} u[n] \cos \left[ \frac{\pi}{N} (n+\frac{1}{2})k  \right]
\]
is adapted to our boundary conditions.

\paragraph{Compute $\proj_{V = \mathcal{I}(U)}$.}
The projection is given for all couples $(U_0,V_0) \in (\mathcal{E}_s, \mathcal{E}_c )$ by 
\[
\proj_{V = \mathcal{I}(U)} (U_0,V_0) = (U^*, \mathcal{I}(U^*))
\]
with $U^*=Q^{-1}(U_0+\mathcal{I}^*V_0)$ and $Q=Id+\mathcal{I}^*\mathcal{I}$.
As $Q$ is a tridiagonal matrix, LU factorization techniques allow to efficiently invert the system.

\subsection{Experiments}

\paragraph{Transport of Gaussian bumps.}

Let us first explore a synthetic case where the initial and final measures have the same mass. This case is shown on Figure \ref{fig:gaussians}. 
The initial and the final measures are both composed of two Gaussian densities of mass $1$ and $2$ supported on the segment $\Om=[0,1]$. The modes of the bumps are located at, from left to right, $0.2$, $0.3$, $0.65$ and $0.9$. The problem is discretized with $N=256$ samples in space and $T=11$ samples in time. The peculiar location of the Gaussian bumps allows to highlight the variations of the behavior of the geodesics, which is highly dependent on the metric and on the parameter $\delta$. 
For the $W_2$ geodesic (Figure \ref{subW2}), the rightmost grey bump is split in half and a part of it travels to the left side of the graph, giving a interpolated measure which, for some applications, is not so natural. The effect of a non-homogeneous functional is visible on Figure \ref{subL2}, where we penalized the $L^2$ norm of the source $\Z$ in the functional. Non-homogeneity makes the density matter: in this case, the source is spread because low densities are favored. 
Consider now the partial transport geodesics: on Figure \ref{subPT1}, we fixed the maximum distance mass can travel to $2\delta=0.3$, while on Figure \ref{subPT2}, that distance is $2\delta=0.15$. We observe, as expected, that the domain $\Om=[0,1]$ is split in an active and an inactive set. Also note that as the $TV$ cost is not strictly convex, geodesics are not unique and what happens in the inactive set depends on the initialization of the algorithm.
Finally, Figures \ref{subWF1} and \ref{subWF2} display the geodesics at $t=1/2$ for the $\WF_{\delta}$ metric with the \emph{cut locus} distances set, respectively, at $\pi \delta=0.3$ and $\pi \delta = 0.15$. Note that there is a slight similarity between the configuration here and the hypotheses of \thref{pairsdiracs}, and the observed behavior can be expected if we extrapolate the conclusions of that Theorem. In the first case, we obtain a geodesic consisting of two travelling bumps which inflate or deflate. While in the second configuration, only one bump travels as $\pi \delta$ is now smaller than the distance between the two bumps at the right. 

\begin{figure}
 \centering
\begin{subfigure}{0.5\linewidth} 
\centering
 \resizebox{1.\linewidth}{!}{
\includegraphics{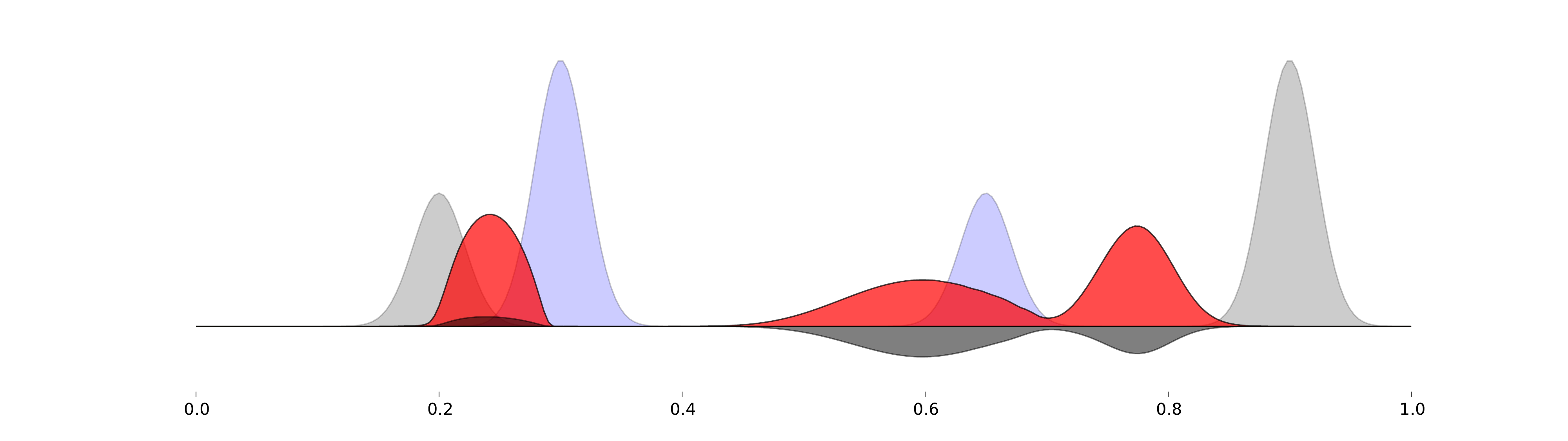}
}
\caption{Standard transport $W_2$}
\label{subW2}
\end{subfigure}%
\begin{subfigure}{0.5\linewidth} 
\centering
  \resizebox{1.\linewidth}{!}{
\includegraphics{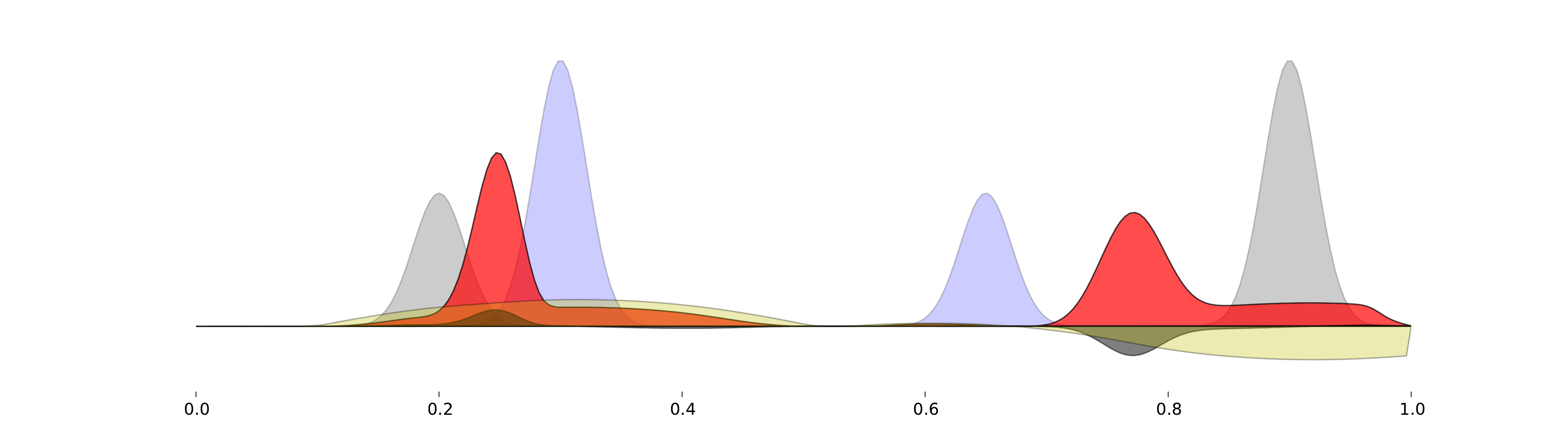}
}
\caption{$L^2$ penalization on the source}
\label{subL2}
\end{subfigure}%

\begin{subfigure}{0.5\linewidth} 
\centering
 \resizebox{1.\linewidth}{!}{
\includegraphics{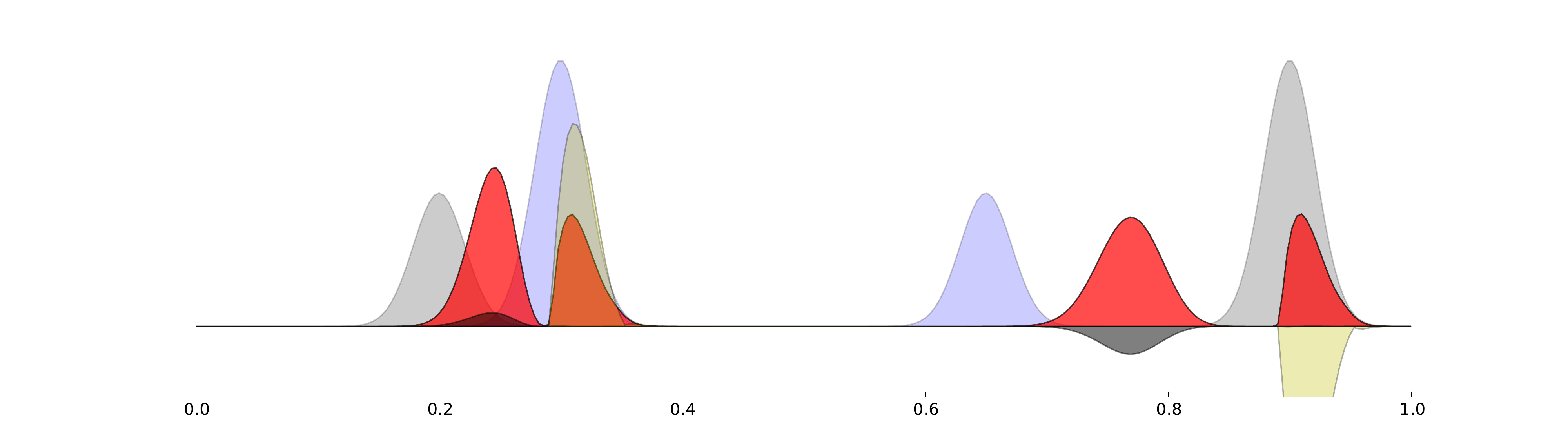}
}
\caption{Partial transport with $c_l=0.3$}
\label{subPT1}
\end{subfigure}%
\begin{subfigure}{0.5\linewidth} 
\centering
 \resizebox{1.\linewidth}{!}{
\includegraphics{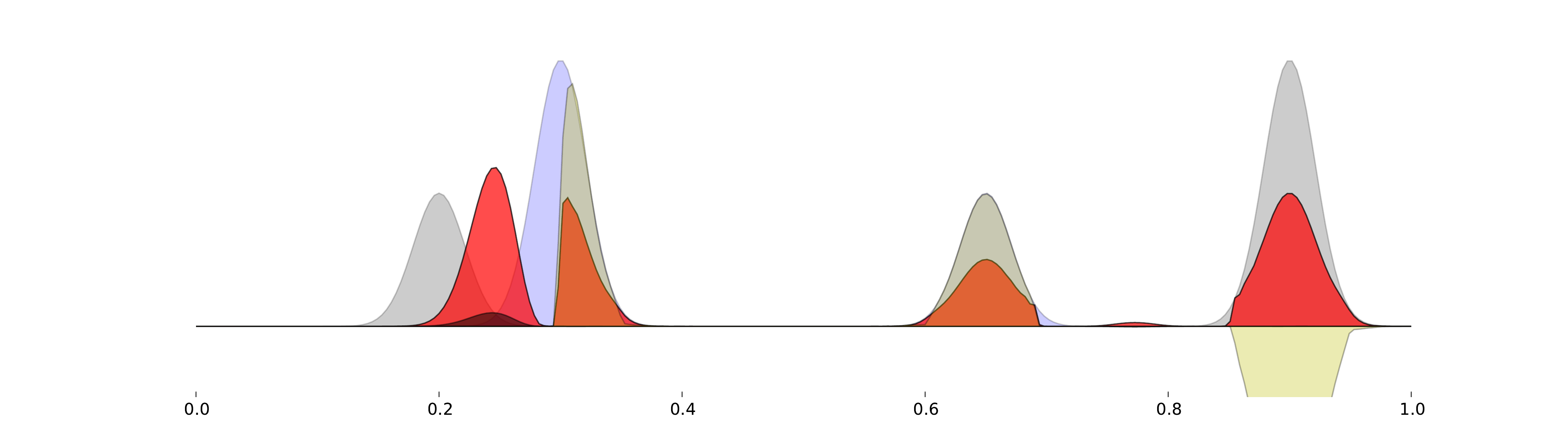}
}
\caption{Partial transport with $c_l = 0.15$}
\label{subPT2}
\end{subfigure}

\begin{subfigure}{0.5\linewidth} 
\centering
  \resizebox{1.\linewidth}{!}{
\includegraphics{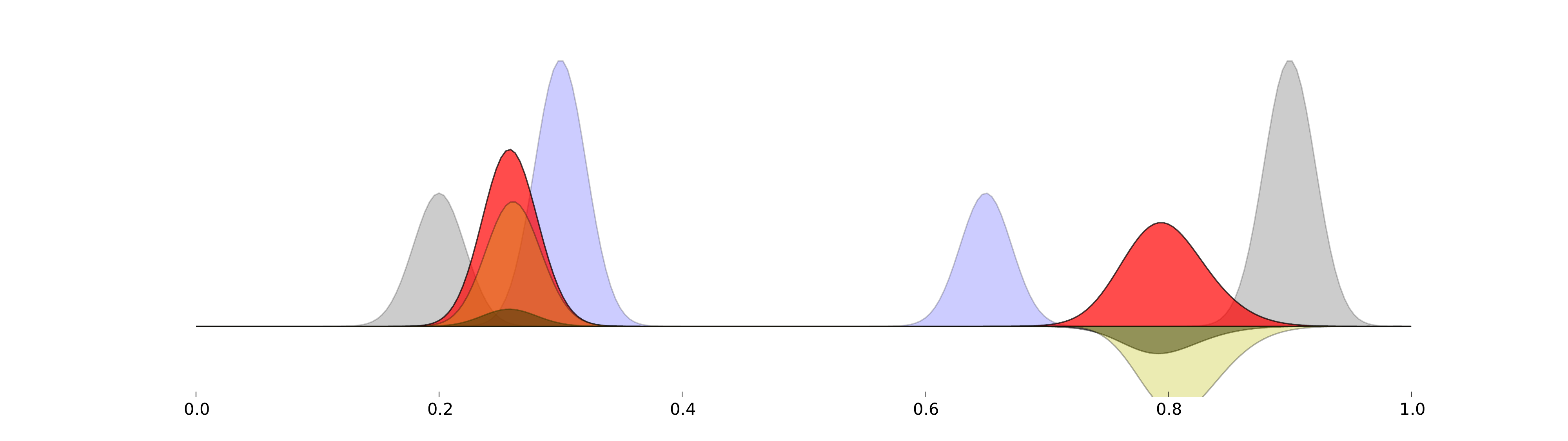}
}
\caption{$\WF_{\delta}$ with $c_l =0.3$}
\label{subWF1}
\end{subfigure}%
\begin{subfigure}{0.5\linewidth} 
\centering
  \resizebox{1.\linewidth}{!}{
\includegraphics{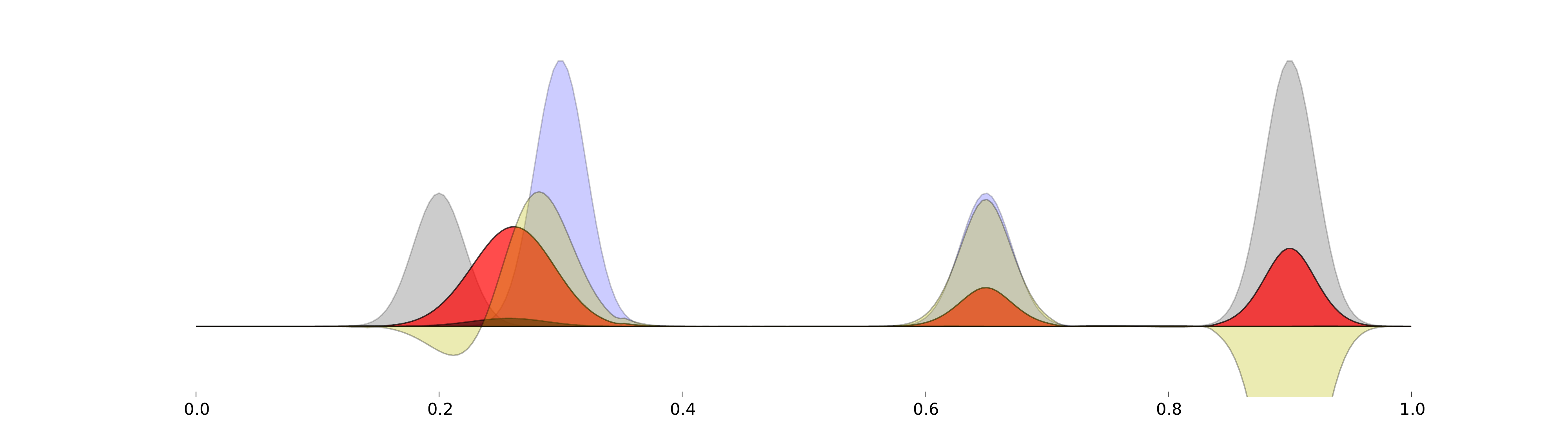}
}
\caption{$\WF_{\delta}$ with $c_l = 0.15$}
\label{subWF2}
\end{subfigure}
\caption{Comparison of interpolations for densities on the segment $[0,1]$. (gray) $\rho_0$ (blue) $\rho_1$ (red) $\rho_{t=1/2}$ (black) $\omega_{t=1/2}$ (yellow) $\zeta_{t=1/2}$. We indicate by $c_l$ the theoretical maximum distance a Dirac can travel, which is a function of $\delta$.}
\label{fig:gaussians}
\end{figure}
\paragraph{Synthetic 2D experiment.}
We now turn our attention to a synthetic example on the domain $\Om=[0,1]^2$. The initial density $\rho_0$ is the indicator of the ring of center $(\frac12, \frac12)$, internal diameter $0.5$ and external diameter $0.7$. The final density $\rho_1$ is the pushforward of $\rho_1$ by a random smooth map and thus $\rho_0(\Om)=\rho_1(\Om)$. The domain is discretized on a centered grid of $64\times 64$ samples in space and $T=12$ samples in time. We compare on Figure \ref{fig: synthetic rho} the geodesics for four different metrics: $d_{FR}$, $W_2$, partial optimal transport (with $2\delta=0.2$) and $\WF_{\delta}$ (with $\delta \pi = 0.4$). Notice that the two first rows thus show the limit geodesics for $\WF_{\delta}$ when $\delta$ tends to $0$ or $+\infty$.

Figure \ref{fig: synthetic details} helps to better understand the geodesics. On the top row, the velocity field shows that the $W_2$ geodesic transports a lot of mass to the bottom left protuberance. On the contrary, metrics allowing local variations of mass (partial transport and $\WF_{\delta}$), attenuate the component of the velocity field which is tangential to the ring, behavior which is more consistent with the intuitive solution. Finally, by looking at the source maps in the bottom row, we see the inactive sets of partial optimal transport with well defined edges and how this behavior strongly differs to that of $\WF_{\delta}$ geodesics.

\begin{figure}
 \centering
  \resizebox{1.0\linewidth}{!}{
\begin{tikzpicture}
\filldraw[color=black!5, fill=black!5] (-1.5,-1.) rectangle (.5,1);
\filldraw[color=black!5, fill=black!5] (11.5,-1.) rectangle (13.5,1);
\node (1r1) at (-0.5,0) {\includegraphics{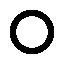}};
\node (3r2) at (2,2.25) {\includegraphics{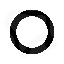}};
\node (3r3) at (4,2.25) {\includegraphics{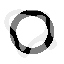}};
\node (3r4) at (6,2.25) {\includegraphics{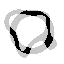}};
\node (3r5) at (8,2.25) {\includegraphics{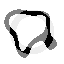}};
\node (3r6) at (10,2.25) {\includegraphics{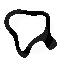}};
\node (0r2) at (2,.75) {\includegraphics{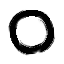}};
\node (0r3) at (4,.75) {\includegraphics{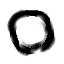}};
\node (0r4) at (6,.75) {\includegraphics{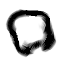}};
\node (0r5) at (8,.75) {\includegraphics{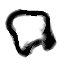}};
\node (0r6) at (10,.75) {\includegraphics{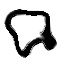}};
\node (1r2) at (2,-.75) {\includegraphics{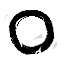}};
\node (1r3) at (4,-.75) {\includegraphics{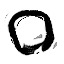}};
\node (1r4) at (6,-.75) {\includegraphics{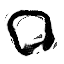}};
\node (1r5) at (8,-.75) {\includegraphics{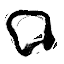}};
\node (1r6) at (10,-.75) {\includegraphics{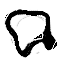}};
\node (2r2) at (2,-2.25) {\includegraphics{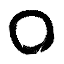}};
\node (2r3) at (4,-2.25) {\includegraphics{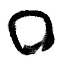}};
\node (2r4) at (6,-2.25) {\includegraphics{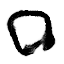}};
\node (2r5) at (8,-2.25) {\includegraphics{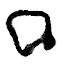}};
\node (2r6) at (10,-2.25) {\includegraphics{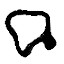}};

\node (1r7) at (12.5,0) {\includegraphics{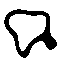}};

\draw[->] 	(1r1)    	-- 	(3r2);
\draw[->] 	(1r1)    	-- 	(0r2);
\draw[->]  	(1r1)    	-- 	(1r2);
\draw[->] 	(1r1)    	-- 	(2r2);
\draw[->]  	(3r6)    	-- 	(1r7);
\draw[->]  	(2r6)    	-- 	(1r7);
\draw[->] 	(0r6)    	-- 	(1r7);
\draw[->]  	(1r6)    	-- 	(1r7);
\draw[->]  	(-.5,-3)    	-- 	(12.5,-3);

\draw[dotted]        (2,0) -- (10,0);
\draw[dotted]        (2,-1.5) -- (10,-1.5);
\draw[dotted]        (2,1.5) -- (10,1.5);
\draw 	(-.5,-3.)  		node {$\bullet$};
\draw 	(-.5,-3.2)  		node {$t=0$};
\draw 	(12.5,-3.2)  	node {$t=1$};
\draw 	(6,-3.2)  	node {$t=0.5$};
\draw 	(-.5,-1.5)  	node {$\rho_0$};
\draw 	(12.5,-1.5)  	node {$\rho_1$};

\end{tikzpicture}
        }
\caption{Geodesics between $\rho_0$ and $\rho_1$ for four metrics (1st row) Fisher-Rao (2nd row) $W_2$ (3rd row) partial optimal transport (4th row) $\WF_{\delta}$}
\label{fig: synthetic rho}
\end{figure}

\begin{figure}[ht]
 \centering
       
\begin{subfigure}{0.3\linewidth} 
\centering
 \resizebox{1.\linewidth}{!}{
\includegraphics[clip,trim=2cm 2cm 2cm 2cm]{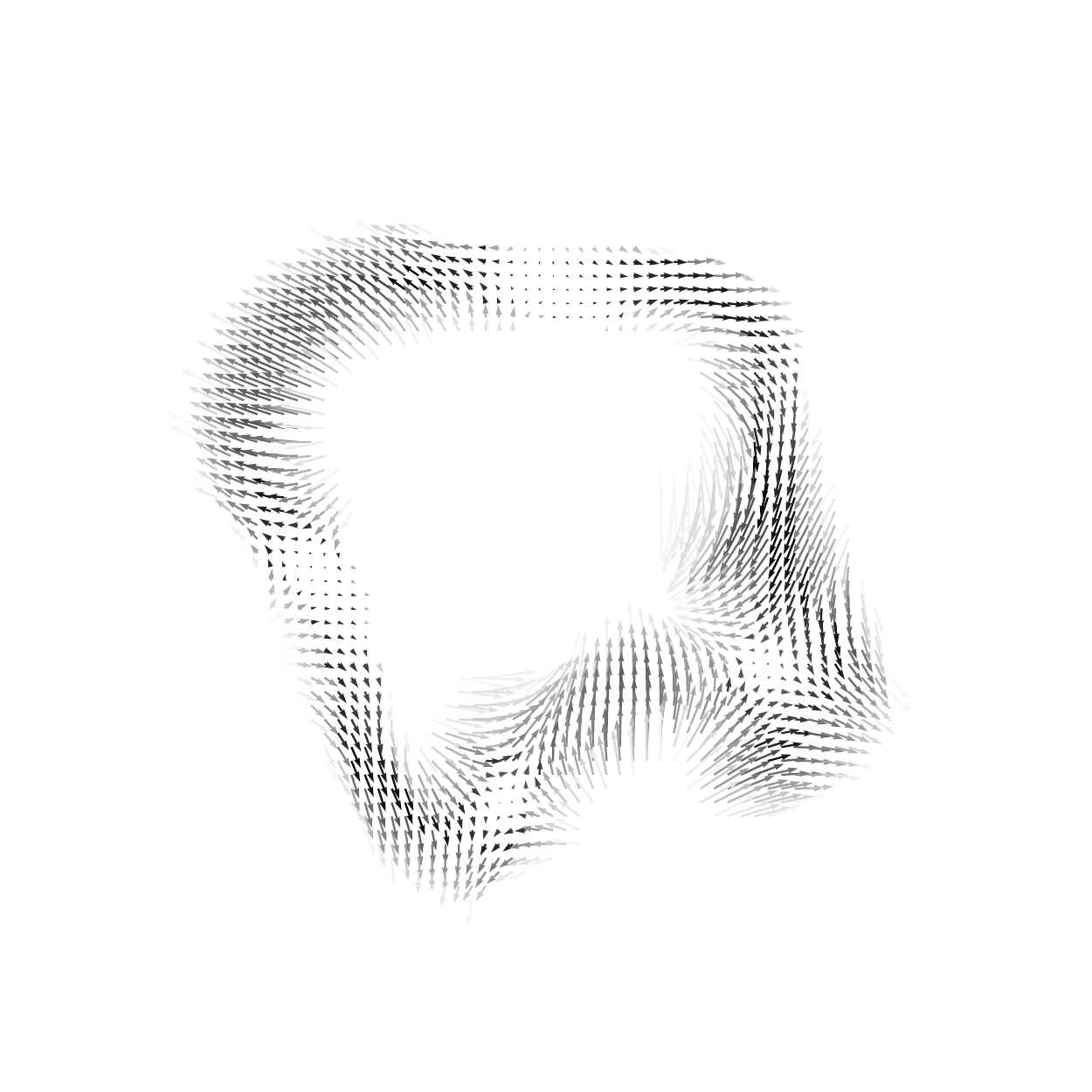}
}
\caption{$W_2$}
\end{subfigure}%
\begin{subfigure}{0.3\linewidth} 
\centering
 \resizebox{1.\linewidth}{!}{
\includegraphics[clip,trim=2cm 2cm 2cm 2cm]{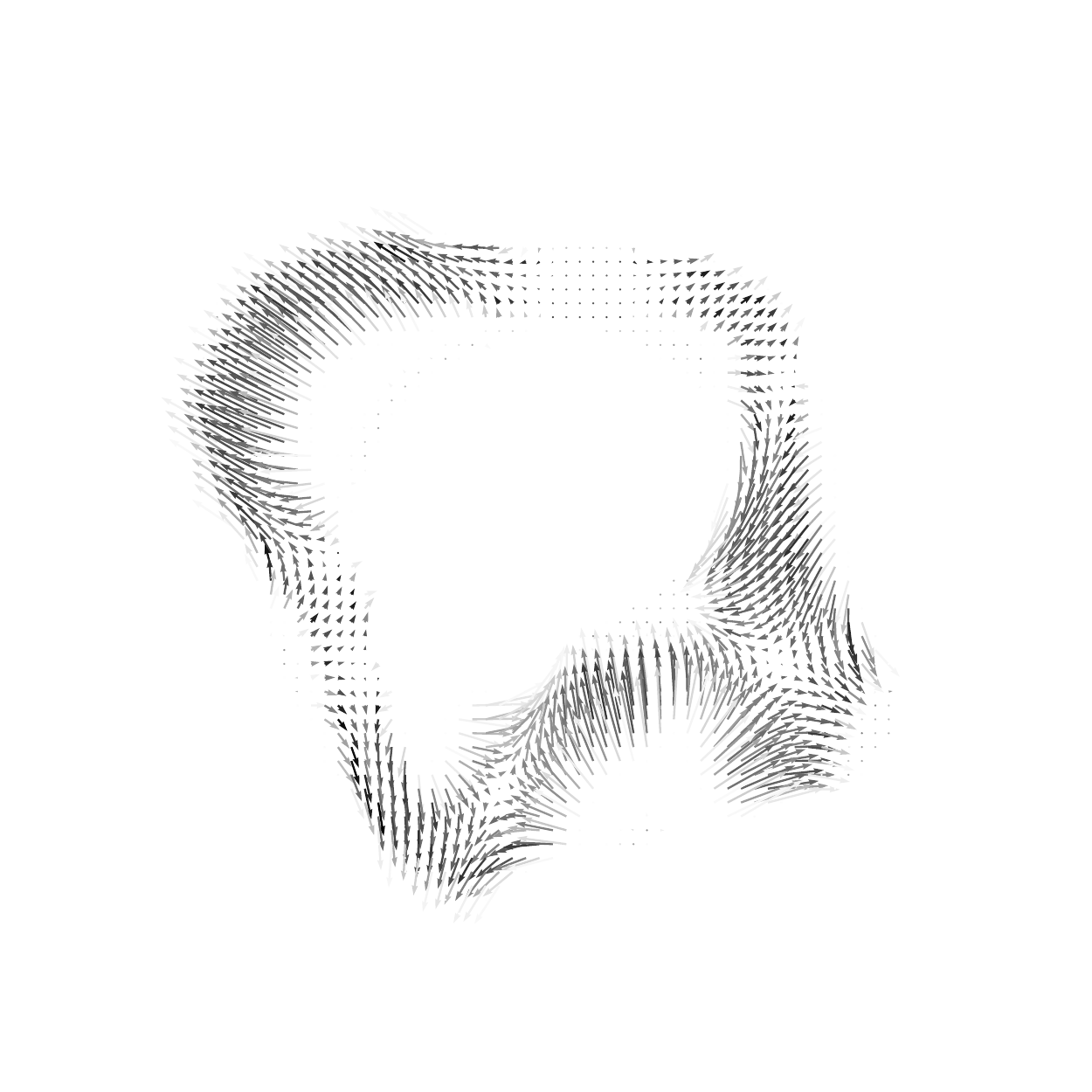}
}
\caption{Partial transport}
\end{subfigure}%
\begin{subfigure}{0.3\linewidth} 
\centering
  \resizebox{1.\linewidth}{!}{
\includegraphics[clip,trim=2cm 2cm 2cm 2cm]{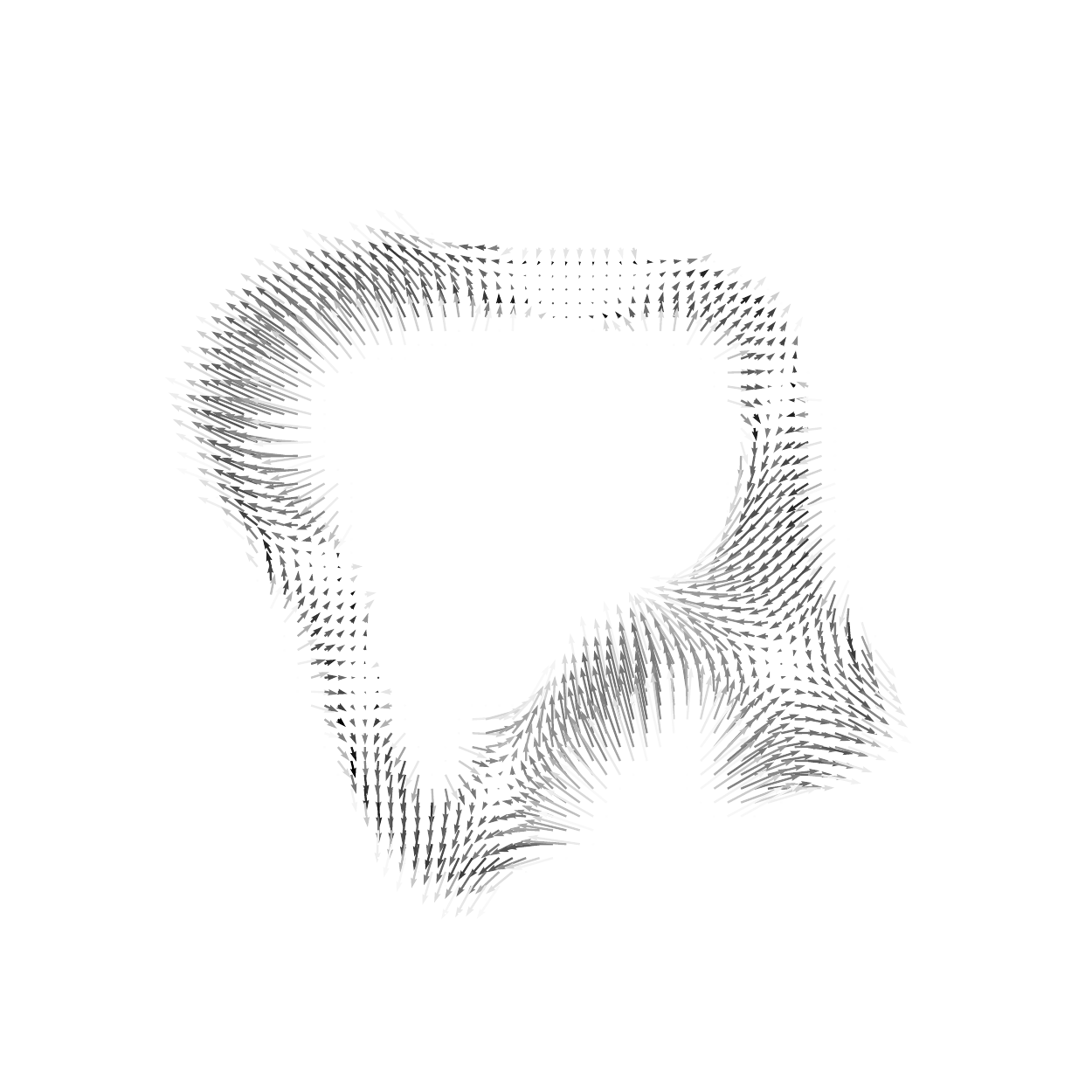}
}
\caption{$\WF_{\delta}$}
\end{subfigure}

\begin{subfigure}{0.25\linewidth} 
\centering
 \resizebox{1.\linewidth}{!}{
\includegraphics[clip]{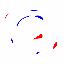}
}
\caption{Partial transport}
\end{subfigure}%
\begin{subfigure}{0.25\linewidth} 
\centering
  \resizebox{1.\linewidth}{!}{
\includegraphics[clip]{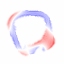}
}
\caption{$\WF_{\delta}$}
\end{subfigure}

\caption{(first row) Velocity field $v = \omega/\rho$ at time $t=1/2$. The higher $\rho$, the darker the arrow. (second row) Source $\zeta$ at time $t=1/2$. Blue stand for negative density and red for positive.}
\label{fig: synthetic details}
\end{figure}


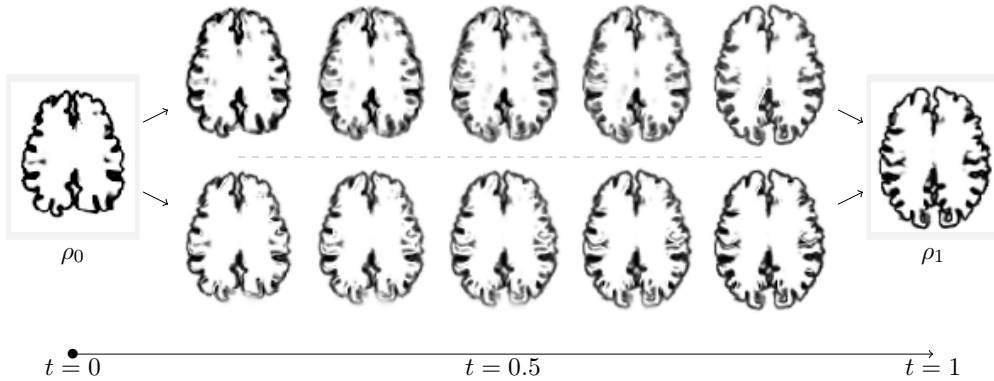
\begin{figure}[ht]
 \centering
  \resizebox{1.0\linewidth}{!}{
\begin{tikzpicture}
\filldraw[color=black!5, fill=black!5] (-1.5,-1.25) rectangle (.5,1.25);
\filldraw[color=black!5, fill=black!5] (11.5,-1.25) rectangle (13.5,1.25);
\node (1r1) at (-0.5,0) {\includegraphics{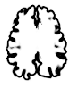}};

\node (0r2) at (2,1.25) {\includegraphics{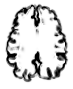}};
\node (0r3) at (4,1.25) {\includegraphics{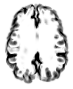}};
\node (0r4) at (6,1.25) {\includegraphics{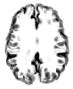}};
\node (0r5) at (8,1.25) {\includegraphics{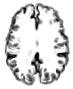}};
\node (0r6) at (10,1.25) {\includegraphics{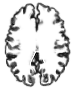}};
\node (1r2) at (2,-1.25) {\includegraphics{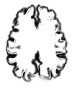}};
\node (1r3) at (4,-1.25) {\includegraphics{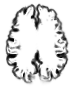}};
\node (1r4) at (6,-1.25) {\includegraphics{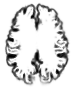}};
\node (1r5) at (8,-1.25) {\includegraphics{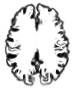}};
\node (1r6) at (10,-1.25) {\includegraphics{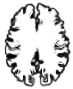}};

\node (1r7) at (12.5,0) {\includegraphics{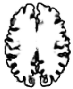}};

\draw[->] 	(1r1)    	-- 	(0r2);
\draw[->]  	(1r1)    	-- 	(1r2);

\draw[dashed,color=black!30] (2,0) --(10,0.);

\draw[->] 	(0r6)    	-- 	(1r7);
\draw[->]  	(1r6)    	-- 	(1r7);
\draw[->]  	(-.5,-3)    	-- 	(12.5,-3);

\draw 	(-.5,-3.)  		node {$\bullet$};
\draw 	(-.5,-3.2)  		node {$t=0$};
\draw 	(12.5,-3.2)  	node {$t=1$};
\draw 	(6,-3.2)  	node {$t=0.5$};
\draw 	(-.5,-1.5)  	node {$\rho_0$};
\draw 	(12.5,-1.5)  	node {$\rho_1$};

\end{tikzpicture}
        }
\caption{Geodesics between $\rho_0$ and $\rho_1$ for two metrics (top row) $W_2$ between rescaled densities (bottom row) $\WF_{\delta}$ metric with $\pi\delta=0.2$.}
\label{fig: brain interp}
\end{figure}

\begin{figure}
 \centering
       
\begin{subfigure}{0.3\linewidth} 
\centering
 \resizebox{1.\linewidth}{!}{
\includegraphics[clip,trim=2cm 1cm 2cm 1cm]{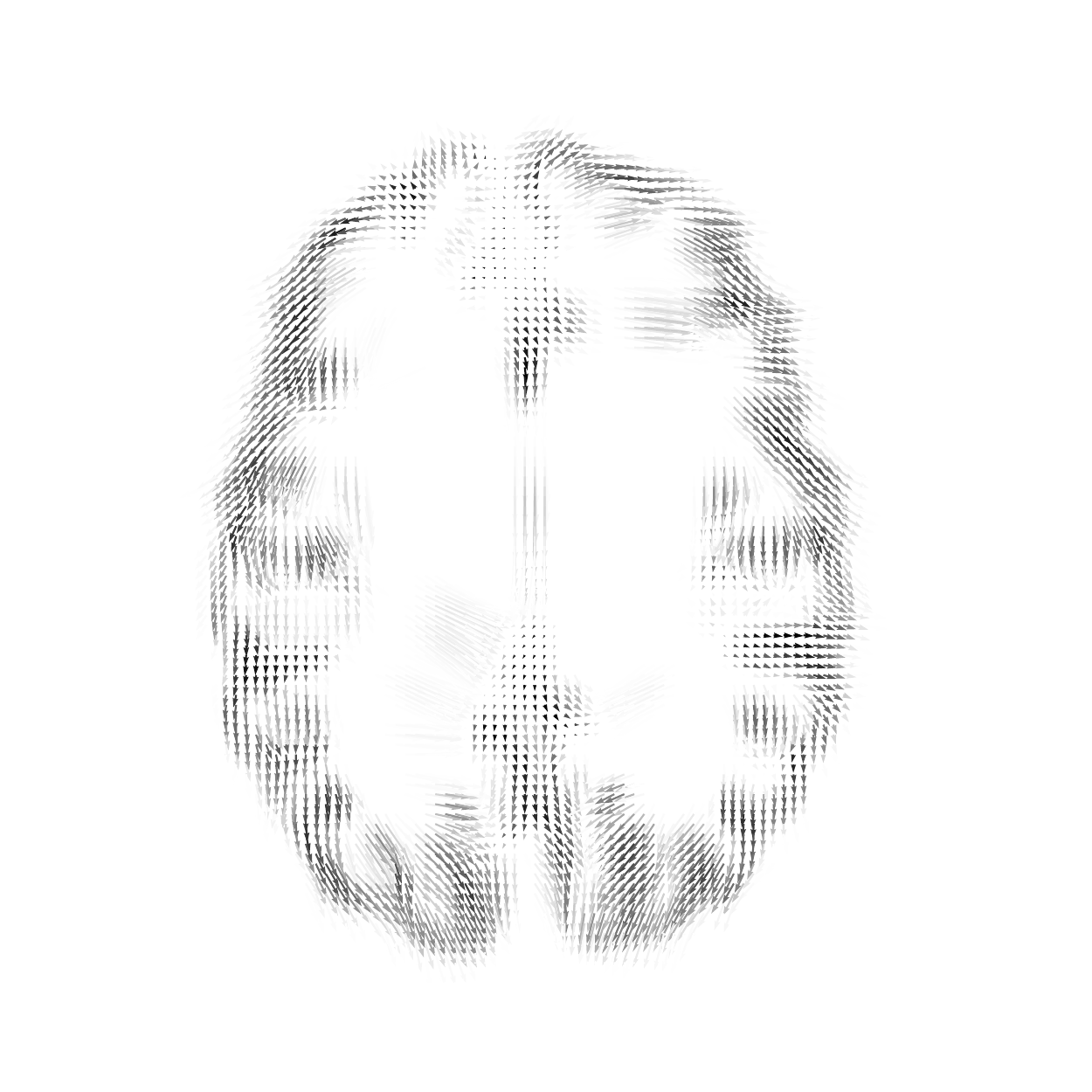}
}
\caption{$v_{0.5}$ for $W_2$}
\label{fig: brain velo w2}
\end{subfigure}%
\begin{subfigure}{0.3\linewidth} 
\centering
 \resizebox{1.\linewidth}{!}{
\includegraphics[clip,trim=2cm 1cm 2cm 1cm]{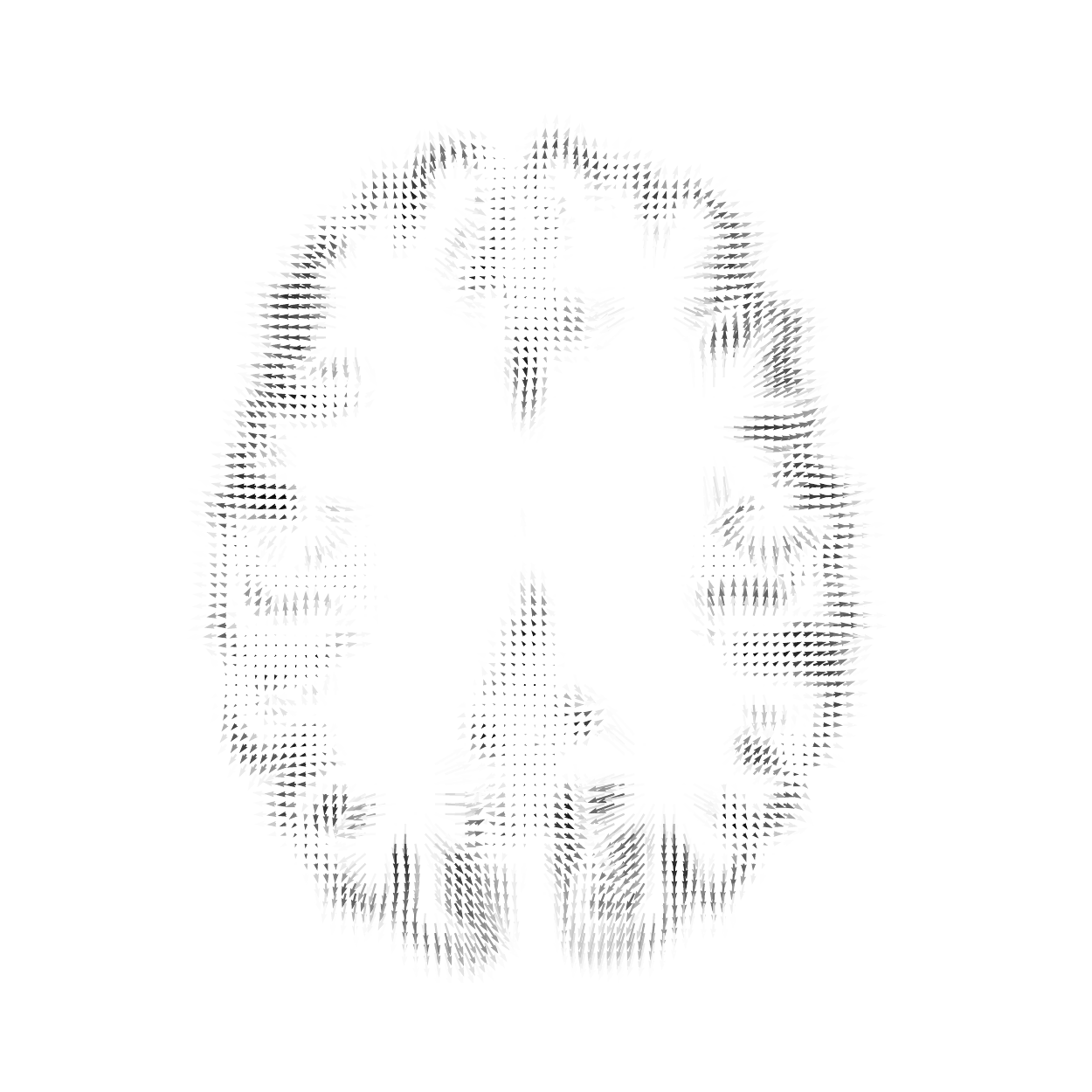}
}
\caption{$v_{0.5}$ for $\WF_{\delta}$}
\label{fig: brain velo wf}
\end{subfigure}%
\begin{subfigure}{0.33\linewidth} 
\centering
  \resizebox{1.\linewidth}{!}{
\includegraphics[clip,trim=0cm -0cm 0cm -0cm]{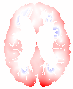}
}
\caption{$g_{0.5}$ for $\WF_{\delta}$}
\label{fig: brain g wf}
\end{subfigure}

\caption{(a) and (b) : velocity field $v=\omega/\rho$ for geodesics at time $t=1/2$. The higher $\rho_{0.5}$, the darker the arrow. (c) rate of growth $g=\Z/\rho$ at time $t=1/2$. Blue stand for negative density and red for positive and $g$ is set to zero when $\rho \approx 0$. Note that theoretically, (b) is the gradient field of (c).}
\label{fig: brain details}
\end{figure}

\paragraph{Biological images interpolation.}
Interpolating between shapes with varying masses is an important motivation for introducing this metric and is the object of our third and last numerical experiment. The initial $\rho_0$ and $\rho_1$ densities are extracted form two segmented images of the same young brain taken at different times. This example is rather challenging for image matching algorithms: matter is creased, folded and grows unevenly. There are $89 \times 73$ samples in space, $12$ samples in time and the spatial domain is $\Om=[0,0.82]\times [0,1]$. We computed the geodesics for $W_2$ and for $\WF_{\delta}$ with $\pi\delta = 0.2$. We show on Figure \ref{fig: brain interp} the geodesics and on Figure \ref{fig: brain details} the velocity fields and the rate of growth a time $t=1/2$.

Most of the tissue growth is located at the bottom of the domain. Consequently, the velocity field associated to the $W_2$ geodesic is dominated by top to bottom components, as seen on Figure \ref{fig: brain velo w2}. To the contrary, the geodesic for the $\WF_{\delta}$ metric manages to locally adapt the rate of growth, and the velocity field is far more consistent with the true evolution of the tissue. However, we retain some artifacts coming from optimal transport: some matter is teared off the brain lining and brought somewhere else to fill a need of mass. This behavior, inherent to the non-smooth nature of optimal transport plans, is clearly observed on Figure \ref{fig: brain interp} near the bulges that appear on the right and left sides of the brain.

Finally, let us recall the reader that, at optimality, the rate of growth $g$ is equal to the dual variable $\varphi$ and thus, the vector field on Figure \ref{fig: brain velo wf} is the gradient of the rate of growth on Figure \ref{fig: brain g wf}, which ranges in $[-0.05,0.18]$.



\section*{Conclusion}

When looking for generalizations of the dynamic optimal transport problem to unbalanced measures, a distance interpolating between optimal transport and Fisher Rao appears naturally. It is the only Riemannian-like deriving from a family of convex homogeneous dynamic minimization problems. 
By analyzing the effect of varying the interpolation factor, and by studying geodesics for atomic measures, we obtained theoretical clues on how the model behaves.
These clues can be summarized by the two following interpretations: 
(i) the distance behaves as a spatially localized version of optimal transport, in the sense that mass whose supports are far away does not interact;
(ii) it also behaves as a regularized version of Fisher-Rao, in the sense that, unlike Fisher-Rao, it is stable by perturbation of the supports. This theoretical behavior is confirmed by the numerical experiments.

\section*{Acknowledgements}

The work of Bernhard Schmitzer has been supported by the Fondation Sciences Math\'ematiques de Paris. 
The work of Gabriel Peyr\'e has been supported by the European Research Council (ERC project SIGMA-Vision). 
We would like to thank Yann Brenier and Jean-David Benamou for stimulating discussions.

\bibliographystyle{alpha}
\bibliography{bibchizat}
\end{document}